\documentclass[12pt]{amsart}
\usepackage[top=1in, bottom=1.25in, left=1in, right=1in]{geometry}
\usepackage{amssymb}
\usepackage{amsmath}
\usepackage{amscd}
\usepackage[pdftex]{graphicx}

\usepackage{xcolor}

\newcommand{\map}[3]{#1\colon #2\to #3}
\newtheorem{theorem}{Theorem}[section]
\newtheorem{lemma}[theorem]{Lemma}
\newtheorem{prop}[theorem]{Proposition}

\newtheorem{cor}[theorem]{Corollary}

\theoremstyle{definition}
\newtheorem{defi}[theorem]{Definition}
\theoremstyle{remark}
\newtheorem{rem}[theorem]{Remark}
\newtheorem{example}[theorem]{Example}
\newcommand{\sym}{\mathrm{Sym}}
\newcommand{\mcg}{\mathrm{Mod}}
\newcommand{\pmcg}{\mathrm{PMod}}
\newcommand{\ppmcg}{\mathrm{PMod}^+}

\newcommand{\bndy}{\mathrm{Bndy}}
\newcommand{\curv}{\mathrm{Curv}}

\newcommand{\forget}[1]{\mathrm{for}_{#1}}
\newcommand{\cut}[1]{\mathrm{cut}_{#1}}

\newcommand{\arcs}{\mathrm{Arcs}}
\newcommand{\odd}{\mathrm{Odd}}

\newcommand{\sq}{\mathrm{sq}}
\newcommand{\sqa}{\mathrm{sq}^\ast}
\newcommand{\nhalf}{\left\lfloor\frac{n}{2}\right\rfloor}

\newcommand{\ctr}{crosscap transposition representation}

\newcommand{\A}{\mathcal{A}}
\newcommand{\B}{\mathcal{B}}
\newcommand{\X}{\mathcal{X}}
\newcommand{\Y}{\mathcal{Y}}

\newcommand{\T}{\mathcal{T}}

\newcommand{\Z}{\mathbb{Z}}
\newcommand{\C}{\mathcal{C}}

\newlength{\customwidth}
\numberwithin{equation}{section}
\title[Geometric representations of the braid group]{Geometric representations of the braid group on a nonorientable surface}
\author{Micha{\l} Stukow \and B{\l}a\.zej Szepietowski}
\address{Institute of Mathematics, Faculty of Mathematics, Physics and Informatics, University of Gda\'nsk, 80-308 Gda\'nsk, Poland} 
\email{michal.stukow@ug.edu.pl}
\email{blazej.szepietowski@ug.edu.pl}
\begin{document}
\begin{abstract}
We classify homomorphisms from the braid group on $n$ strands to the pure mapping class group of a nonoriantable surface of genus $g$.
For $n\ge 14$ and $g\le 2\lfloor{n/2}\rfloor+1$ every such homomorphism is either cyclic, or it maps standard generators of the braid group  to either distinct Dehn twists, or distinct crosscap transpositions, possibly multiplied by the same element  of the centralizer of the image.   
\end{abstract}

\maketitle
\section{Introduction}
Let $S$ be a compact surface, orientable or nonorientable, possibly disconnected and with boundary. We define the mapping class group $\mcg(S)$ of $S$ to be the group of isotopy classes of all homeomorphisms of $S$. If $S$ is orientable then $\mcg(S)$ is usually called the extended mapping class group. The pure mapping class group   $\pmcg(S)$ of $S$ is the subgroup of $\mcg(S)$ consisting of elements preserving every component of $S$ and every component of $\partial S$. For orientable $S$ we denote by $\mcg^+(S)$ (respectively $\ppmcg(S)$) the subgroup of $\mcg(S)$ (resp. $\pmcg(S)$) consisting of orientation preserving mapping classes. 
The relative mapping class group $\mcg(S,\partial S)$ of a surface with nonempty boundary is the group of isotopy classes of  homeomorphisms of $S$ fixing every point of $\partial{S}$.

A connected orientable (resp. nonorientable) surface of genus $g$ with $b$ boundary components will be denoted by $S_{g,b}$ (resp.  $N_{g,b}$). If $b=0$ we drop it from the notation. Note that $N_{g,b}$ is obtained from $S_{0,g+b}$ by gluing $g$ M\"obius bands (also called crosscaps) along $g$ distinct boundary components of $S_{0,g+b}$ (Figure~\ref{r01}).

%Review of the results of Castel and CM.
A \emph{geometric representation} of a group $G$ is any homomorphism from $G$ to $\mcg(S)$ for some surface $S$. A classification of representations of a given group $G$ in $\mcg(S)$ provides an information  which could be useful in the study of homomorphisms between mapping class groups. A great example is the work of Castel \cite{Castel} on representations of the braid group $\B_n$ in $\pmcg(S_{g,b})$ and its generalization due to Chen and  Mukherjea \cite{CM}.
Castel \cite{Castel} proved that for $n\ge 6$, $g\le n/2$ and $b\ge 0$ every homomorphism 
$\rho\colon\B_n\to\ppmcg(S_{g,b})$ is either cyclic or a transvection of the standard twist representation (defined below). This classification was extended in \cite{CM} to the case $g<n-2$ but in a reduced range, $n\ge 23$.

In this paper we generalize the results of \cite{Castel,CM} to the case of a nonorientable surface. That is we classify representations of $\B_n$ in $\pmcg(N_{g,b})$ for $n\ge 14$ and $g\le 2\left\lfloor n/2\right\rfloor+1$.
Our classification includes \emph{the \ctr}, defined in \cite{SzepPM}, which does not exist in the case of an orientable surface. 

 We denote the standard generating set of $\B_n$ by $\{\sigma_i\colon i=1,\dots,n-1\}$.
 The defining relations of $\B_n$ are
 \[\sigma_i\sigma_j=\sigma_j\sigma_i\quad\textrm{if\ } |i-j|\ne 1,\qquad \sigma_i\sigma_j\sigma_i=\sigma_j\sigma_i\sigma_j\quad\textrm{if\ } |i-j|=1.\]
\subsection{The standard twist representation}
%Definition of the monodromy representation of $\B_n$ in $\pmcg(S_{g,b})$ and $\pmcg(N_{g,b})$.
By a \emph{curve} in $S$ we understand a simple closed curve. 
A curve is either two- or one-sided depending on whether its regular neighbourhood is an annulus or a M\"obius band respectively. All curves in an orientable surface are two-sided.
We denote by $I(a,b)$ the geometric intersection number of two curves $a,b$. A sequence $C=(a_1,\dots,a_{n-1})$ of two-sided curves in $S$ is called \emph{a chain of nonseparating curves} if 
$I(a_i,a_{i+1})=1$ for $1\le i\le n-2$ and $I(a_i,a_j)=0$ for $|i-j|>0$. If we fix an orientation of a regular neighbourhood of the union of the curves $a_i$, then $C$ determines \emph{the standard twist representation} $\rho_C\colon\B_n\to\pmcg(S)$ defined by \[\rho_C(\sigma_i)=t_{a_i},\quad i=1,\dots,n-1,\] where $t_{a_i}$ is the right-handed Dehn twist about $a_i$ with respect to the fixed orientation (Figure~\ref{r01}).
\begin{figure}[h]
\begin{center}
\includegraphics[width=0.8\customwidth]{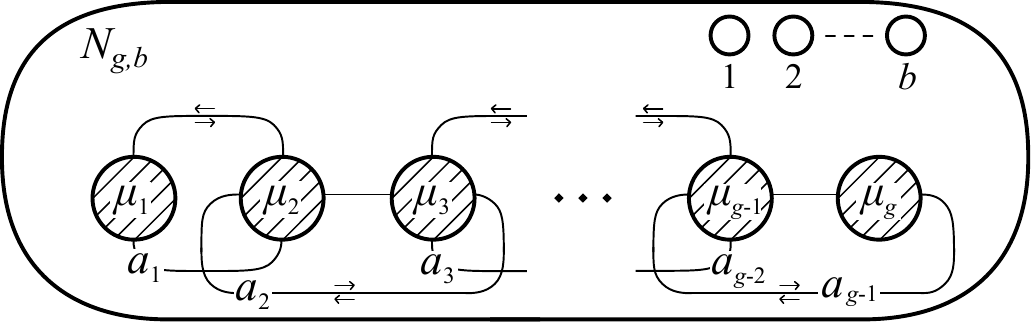}
\caption{Standard chain of nonseparating curves in $N_{g,b}$.}\label{r01} %
\end{center}
\end{figure}
\begin{rem}\label{rem:g_chain}
The chain of nonseparating curves from Figure \ref{r01} determines a standard twist representation $\rho_C\colon \B_g\to\pmcg(N_{g,b})$ for any $g\ge 3$ and $b\ge 0$. If $g$ is odd and $g\ge 5$, then this chain can be extended by adding a curve $a_g$ passing once through each of first $g-1$ crosscaps. This extended chain determines a standard twist representation $\rho_{C'}\colon \B_{g+1}\to\pmcg(N_{g,b})$.    
\end{rem}

\subsection{The \ctr}\label{ctr_def} 
Now let $N=N_{g,b}$ be a nonorientable surface.  
 A sequence $C=(a_1,\dots,a_{n-1})$ of separating curves in $N$ is called \emph{a chain of separating curves} if
\begin{enumerate}
\item $a_i$ bounds a one-holed Klein bottle  for $i=1,\dots,n-1$,
\item $I(a_i,a_{i+1})=2$ for $i=1,\dots,n-2$,
\item $I(a_i,a_j)=0$ for $|i-j|>1$.
\end{enumerate} 
Let $K_i$ be the one-holed Klein bottle bounded by $a_i$.
Note that $K_i\cap K_{i+1}$ is a M\"obius strip for $i=1,\dots,n-2$, and we denote its core curve by $\mu_{i+1}$. We also let $\mu_1$ and $\mu_{n}$ be the core curves of  $K_{1}\setminus K_{2}$ and $K_{n-1}\setminus K_{n-2}$ respectively (Figure \ref{r02}).
\begin{figure}[h]
\begin{center}
\includegraphics[width=0.8\customwidth]{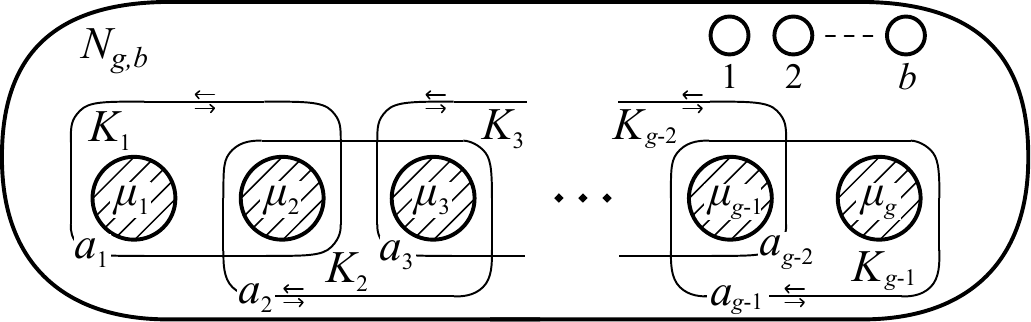}
\caption{Standard chain of separating curves in $N_{g,b}$.}\label{r02} %
\end{center}
\end{figure}
Fix an orientation of a regular neighbourhood of the union of $a_i$ for $i=1,\dots,n-1$ and let $t_{a_i}$ be 
the right-handed Dehn twist about $a_i$. Let $u_i$ be the \emph{crosscap transposition} supported in $K_i$, swapping $\mu_i$ and $\mu_{i+1}$, and such that $u_i^2=t_{a_i}$ (Figure \ref{r03}). 
It is proved in \cite{SzepPM} that the mapping
\[\theta_C(\sigma_i)=u_i,\quad i=1,\dots,n-1,\]
defines a homomorphism $\theta_C\colon\B_n\to\pmcg(N)$ called  
\emph{\ctr}. If $(g,b)\ne (n,0)$ then $\theta_C$ is injective.
\begin{figure}[h]
\begin{center}
\includegraphics[width=0.88\customwidth]{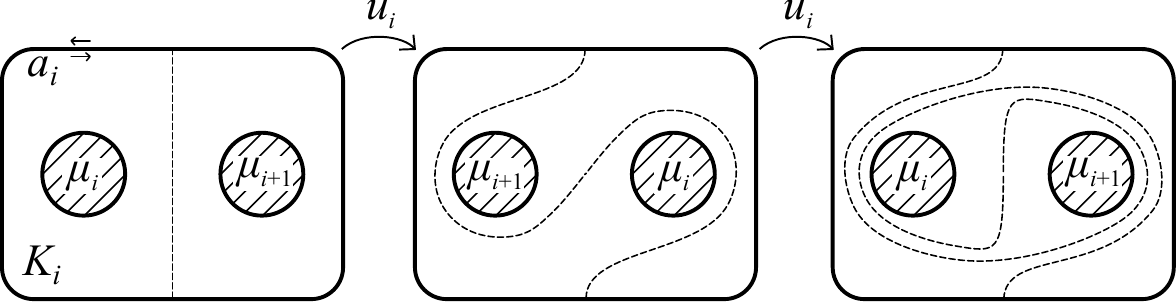}
\caption{Crosscap transposition $u_i$.}\label{r03} %
\end{center}
\end{figure}

%Definition of a transvection.
Given a homomorphism $\rho\colon\B_n\to\pmcg(S)$ and an element $\tau\in\pmcg(S)$ such that $\tau$ commutes with $\rho(\sigma_i)$ for $1\le i\le n-1$, we define a homomorphism  $\rho^\tau\colon\B_n\to\mcg(S)$, called \emph{a transvection} of $\rho$, by 
 \[\rho^\tau(\sigma_i)=\tau\rho(\sigma_i),\quad i=1,\dots,n-1.\]
A homomorphism  $\rho\colon\B_n\to\pmcg(S)$ is called \emph{cyclic} if $\rho(\B_n)$ is a cyclic group. Note that a cyclic homomorphism is a transvection of the trivial one.

\subsection{The main theorems}

\begin{theorem}\label{thm:main_pmcg}
Let $n\ge 14$ and let $N=N_{g,b}$ with  $g\le 2\left\lfloor n/2\right\rfloor+1$ and $b\ge 0$. Then any homomorphism $\rho\colon\B_n\to\pmcg(N)$ is either cyclic, or is a transvection of a standard twist representation, or is a transvection of a \ctr.
\end{theorem}

\begin{theorem}\label{thm:main_relmcg}
Theorem \ref{thm:main_pmcg} still holds when $\pmcg(N)$ is replaced by $\mcg(N,\partial N)$.
\end{theorem}

\begin{rem}\label{rem:existence}
 Let $n\ge 5$ and $N=N_{g,b}$. 
 \begin{enumerate}
     \item A {\ctr} of $\B_n$ in $\pmcg(N)$ (resp. $\mcg(N,\partial N)$) exists if and only if $g\geq n$.  
     \item It is not difficult to see that for $g\ge 4$ the maximum length of a chain of nonesepating curves in $N$ is 
     $g-1$ if $g$ is even, and $g$ if $g$ is odd. It follows that a standard twist representation of $\B_n$ in $\pmcg(N)$ (resp. $\mcg(N,\partial N)$) exists if and only if
     \[g\ge2\left\lfloor\frac{n-1}{2}\right\rfloor+1
     =\begin{cases}
     n-1 & \textrm{if\ $n$\ is\ even,}\\
     n & \textrm{if\ $n$\ is\ odd.}
     \end{cases}
          \]
 \end{enumerate}
\end{rem}

 Theorems \ref{thm:main_pmcg} and \ref{thm:main_relmcg} can be applied to classify homomorphisms between mapping class groups of nonorientable surfaces. In this paper we give a simple example of such application.

Let $\T(N)$ (resp. $\T(N,\partial{N})$)  denote the index $2$ \emph{twist subgroup} of $\ppmcg(N)$ (resp. $\mcg(N,\partial N)$) generated by all Dehn twists \cite{StukowOJM}. 
\begin{theorem}\label{thm:app}
    For $g\ge 13$, $b,b'\ge 0$ and $g'\le2\left\lfloor\frac{g-1}{2}\right\rfloor$, every homomorphism from $\T(N_{g,b})$ or $\T(N_{g,b},\partial N_{g,b})$  to $\pmcg(N_{g',b'})$ is trivial.
\end{theorem}
For closed surfaces a similar result is proved in \cite{SzepAGT}.

%\subsection{Comparison with the work of Castel and Chen-Mukherjea}
\subsection{Comparison with the case of an orientable surface}
Similarly as in \cite{Castel, CM}, our major tool is the Nielsen-Thurston theory and the canonical reduction system, denoted by $\C(f)$ for a mapping class $f$, introduced by Birman, Lubotzky and McCarthy \cite{BLM}. We also use the totally symmetric sets of Kordek 
 and Margalit \cite{KM}. Our proof of Theorem \ref{thm:main_pmcg} follows \cite{CM} rather closely, but it also contains some significant innovations resulting from differences between the mapping class groups of orientable and nonorientable surfaces.

%{\bf Crosscap transpositions versus Dehn twists.} 
The mapping class group of a nonorientable surface of genus at least $2$ contains crosscap transpositions, which do not exist in the case of an orientable surface. They 
are elementary elements of $\mcg(N_{g,n})$ playing a similar role to Dehn twists and satisfying similar algebraic relations (see \cite{LSz,PSz}). However, geometric implications of relations between crosscap transpositions are not as well understood as in the case of Dehn twists. In particular, the braid relation \[t^k_at^l_bt^l_a=t^l_bt^k_at^l_b\] between non-zero powers of distinct Dehn twists $t_a$ and $t_b$ implies $I(a,b)=1$ and $|k|=|l|=1$ (see \cite{StukowFM}). This property of twists is used in our proof of Theorem~\ref{thm:main_pmcg} (see Proposition \ref{prop:braid_on_N31_nonsep}), as well as in \cite{Castel, CM}, but there is no analogous result for crosscap transpositions in the literature. Therefore, in order to prove that a given homomorphism $\rho\colon\B_n\to\pmcg(N_{g,n})$ 
is a transvection of a crosscap transposition representation, we constrain the intersection of the canonical reduction systems $\C(\rho(\sigma_i))$ corresponding to consecutive generators of $\B_n$ by a geometric argument, which requires the assumption $g\le 2\left\lfloor n/2\right\rfloor+1$.

 One feature in which the {\ctr} $\theta_C$ defined in Section \ref{ctr_def} differs from a transvection of a standard twist representation is that
$\C(\theta_C(\sigma_i))$ consists in a single \emph{separating} curve for $i=1,\dots,n-1$. To be more specific, if $C=(a_1,\dots,a_{n-1})$ is a chain of separating curves determining $\theta_C$ then $\C(\theta_C(\sigma_i))=\{a_i\}$. On the other hand, if $\rho$ is a transvection of a standard twist representation, then $\C(\rho(\sigma_i))$ contains no separating \emph{special} components (a component of $\C(\rho(\sigma_i))$ is special if it is not contained in any other $\C(\rho(\sigma_j))$, $j\ne i$). In the case of a nonorientable surface we have to consider both separating and non-separating special components of $\C(\rho(\sigma_i))$, which makes the process of reconstructing the chain of curves from the algebra of the braid group considerably more complicated -- see Section \ref{sec:separating}.

%The main difference between our work and \cite{Castel,CM} is related to the fact that the {\ctr} $\theta_C$ does not exist in the case of an orientable surface. 

%{\bf One-sided versus two-sided curves.}
On a nonorientable surface, the canonical reduction system $\C(f)$ of a mapping class $f$ is in general more complicated than in the case of an orientable surface, as it may contain one-sided and two-sided curves. 
Similarly as in \cite{Castel, CM}, an important step in the proof of Theorem~\ref{thm:main_pmcg} consists in showing that if $\rho(\B_n)$ preserves some two-sided multicurve $\A$, then the action of $\rho(\B_n)$ on the components of $\A$ is cyclic (Proposition \ref{prop:ch8:8.1}). However, this fact is not true for multicurves with one-sided components. For example, the {\ctr} $\theta_C$ defined above preserves the multicurve $\A=\{\mu_1,\dots,\mu_n\}$ consisting of one-sided curves, and the induced homomorphism $\B_n\to\sym(\A)$ is surjective. Therefore, to avoid issues with one-sided curves, instead of $\C(f)$ we use the subset of its two-sided components, denoted by $\C^+(f)$.

If $\C^+(f)\ne\C(f)$, then $\C^+(f)$ is not a complete reduction system for $f$, which leads us to introducing the notion of an 
\emph{almost pseudo-Anosov} mapping class. By definition, $f$ is almost pseudo-Anosov if $\C(f)\neq\emptyset$ but $\C^+(f)=\emptyset$. We show that an almost pseudo-Anosov mapping class $f$ has an important property of pseudo-Anosov mapping classes, namely the centraliser of $f$ is virtually cyclic (Proposition \ref{apA:virtuallyCyclic}). 

%Working with surfaces with boundary throughout the whole proof.
%Another difference between our paper and \cite{CM} is that we allow the surface to have boundary throughout the whole proof, whereas Chen and Mukherjea first work with a closed surface. Then, at one point of their proof (namely in Sec. 10.3 of \cite{CM}) they use the natural projection $\pmcg(S_{g,b})\to\mcg(S_g)$ to obtain information about $\C(f)$ for $f\in\pmcg(S_{g,b})$ by looking at its representative in $\mcg(S_g)$. We find this particular step of their proof problematic, for the reasons explained in Section \ref{CMbug}. By pointing out this issue we don't question correctness of the result of Chen and Mukherjea, only accuracy of their argument. In fact, our paper shows that the methods of \cite{CM}, such as their proof of Proposition 3.1 corresponding to our Proposition \ref{ch3:curves:AIC}, can be modified to work also for surfaces with boundary.

%\subsection{About the bounds on $n$ and $g$.}

\subsection{Outline of the paper.}
In Section \ref{sec:preli} we review a theorem of Birman and Chillingworth and canonical reduction systems.  We also introduce the notion of an almost pseudo-Anosov mapping class and establish its fundamental property.
In Section \ref{sec:total_sym} we review the notions of a totally symmetric set and a totally symmetric multicurve. 
Section \ref{sec:start} contains the outline of the proof of Theorem \ref{thm:main_pmcg} which occupies Sections \ref{sec:irred} through \ref{sec:finish}.
In Section \ref{sec:app} we deduce Theorem \ref{thm:main_relmcg} from Theorem \ref{thm:main_pmcg} and show how it implies Theorem \ref{thm:app}.
%%%%%%%%%%%%%%%%%%%%%%%%%%%%%%%%%%%%%%%
%%%%%%%%%%%%%%%%%%%%%%%%%%%%%%%%%%%%%%
\section{Preliminaries}\label{sec:preli}
\subsection{The subgroup $\ppmcg(N)$.}
For a nonorientable surface with boundary $N=N_{g,b}$, we denote by $\ppmcg(N)$ the subgroup of $\pmcg(N)$ consisting of elements preserving orientation of each component of $\partial N$. We have the following exact sequences.
\begin{equation*}%\label{es:ppmcg}
1\to\ppmcg(N)\to\pmcg(N)\to\Z_2^b\to 1,    
\end{equation*}
\begin{equation*}%\label{es:forget}
1\to\langle t_d\mid d\in\bndy(N)\rangle\to\mcg(N,\partial N)\stackrel{\forget{\partial N}}{\longrightarrow}\ppmcg(N)\to 1,
\end{equation*}
where $\forget{\partial N}$ is \emph{the forgetful homomorphism} and  $\bndy(N)$ denotes the set of components of $\partial{N}$.

\subsection{Birman-Chillingworth theorem and its generalizations.}
Let $p\colon S\to N$ be the orientation double covering, where $N=N_{g,b}$ and $S=S_{g-1,2b}$. Let $\tau\colon S\to S$ be the covering involution (orientation reversing involution without fixed points).
For any homomorphism $f\colon N\to N$  let $\widetilde{f}\colon S\to S$  be the unique orientation preserving lift of $f$.  The following theorem was proved in \cite{BirChil} for $b=0$ and $g\ge 3$, and in \cite{GGM} for $b\ge 1$ and $g\ge 1$.
\begin{theorem}\label{th:BirChil}
For $(g,b)\ne(2,0)$ the mapping $\theta(f)=\widetilde{f}$ defines an injective homomorphism $\theta\colon\mcg(N_{g,b})\to\mcg^+(S_{g-1,2b})$.
\end{theorem} 
\begin{lemma}\label{lem:lift_ppmod}
Let  $\theta\colon\mcg(N)\to\mcg^+(S)$ be the homomorphism from Theorem \ref{th:BirChil}. Then $\theta^{-1}(\ppmcg(S))=\ppmcg(N)$.
\end{lemma}
\begin{proof}
Let $c$ be an oriented boundary curve of $S$ and $f\in\mcg(N)$. We claim that $\widetilde{f}$ fixes $c$ if and only if $f$ fixes $p(c)$ (which is an oriented boundary curve of $N$).
We have $f(p(c))=p(\widetilde{f}(c))$. Clearly, if $\widetilde{f}(c)=c$ then $f(p(c))=p(c)$. Conversely, suppose that $p(c)=f(p(c))=p(\widetilde{f}(c))$. Then $\widetilde{f}(c)=c$ or
 $\widetilde{f}(c)=\tau(c)$. However, the latter is impossible because $\tau\circ\widetilde{f}$ is reversing orientation of $S$, and hence also of $c$. 
 \end{proof}
\begin{cor}\label{cor:pureBirChil}
For $(g,b)\ne(2,0)$ the mapping $\theta(f)=\widetilde{f}$ defines an injective homomorphism $\theta\colon\ppmcg(N_{g,b})\to\ppmcg(S_{g-1,2b})$.
\end{cor}
%The following relative version is proved in \cite[Lemma 3.1]{SzepPM}.
%\begin{theorem}\label{th:relBirChil}
%For $g+2b\ge 3$ the mapping $\theta(f)=\widetilde{f}$ defines an injective homomorphism $\theta\colon\mcg(N,\partial N)\to\mcg(S,\partial S)$.
%\end{theorem} 

Let us record two consequences of Theorem \ref{th:BirChil} which will be used later.
\begin{lemma}\label{lem:smallg}
Let $n\ge 6$ and $b\ge 0$. 
\begin{enumerate}
    \item For $g<\frac{n}{2}-1$ every representation of $\B_n$ in $\ppmcg(S_{g,b})$ is cyclic.
    \item For $g<\frac{n}{2}$ every representation of $\B_n$ in $\ppmcg(N_{g,b})$ is cyclic. 
\end{enumerate}
\end{lemma}
\begin{proof}
    The first statement follows from \cite[Theorem 1]{Castel}. The second statement follows from the first one and Corollary \ref{cor:pureBirChil}.
\end{proof}

\begin{lemma}\label{lem:izoS0N1}
\begin{enumerate}
    \item Let $f\in\mcg^+(S_{0,b})$ be a periodic mapping class. If $f$ fixes at least three boundary components of $S_{0,b}$, then $f$ is the identity.
    \item Let $f\in\mcg(N_{1,b})$ be a periodic mapping class. If $f$ fixes at least two boundary components of $N_{1,b}$ and preserves their orientation, then $f$ is the identity.
\end{enumerate}
\end{lemma}
\begin{proof}
Part (1) is well known, see for example \cite[Corollary A.3.5]{Castel}. For part (2) consider the embedding  $\theta\colon\mcg(N_{1,b})\to\mcg^+(S_{0,2b})$  from Theorem \ref{th:BirChil}. If $f$ fixes two boundary components of $N_{1,b}$ and preserves their orientation, then $\theta(f)$ fixes four boundary components of $S_{0,2b}$ by the proof of Lemma \ref{lem:lift_ppmod}. By part (1), $\theta(f)$ is then the identity, and so is $f$.
\end{proof}

\subsection{Multicurves} \label{sub:sec:multicurves}
A simple closed curve in $N$ is \emph{trivial} if it bounds a disc or a M\"obius band.
We denote by $\curv(N)$ the set of isotopy classes of unoriented nontrivial simple closed curves in $N$ which are not isotopic to a boundary component.
The subset of $\curv(N)$ consisting of two-sided curves is denoted by $\curv^+(N)$.
A subset $\A=\{a_1,\dots,a_k\}$ of $\curv(N)$ is \emph{a multicurve} if $I(a_i,a_j)=0$ for $a_i\ne a_j$. If $\A\subset \curv^+(N)$, we say that $\A$ is \emph{two-sided}.
We denote by $N_\A$ the subsurface obtained by removing from $N$ an open regular neighbourhood of the union of the components of $\A$. 
Elements of $\bndy(N_\A)\cap\bndy(N)$ are called \emph{external boundary components} of $N_A$. If $\A$ is two-sided then every non-external component of $\partial N_\A$ is isotopic in $N$ to a component of $\A$. 
% We distinguish two types of boundary components of $N_\A$: those resulting from cutting along $\A$ and those belonging to $\partial N$. we call the latter \emph{external boundary components}.

 Let $\mcg_\A(N)$ be the stabiliser of $\A$ in $\mcg(N)$ (the elements of $\mcg_\A(N)$ may permute the components of $\A$). There is a natural homomorphism
 \[\cut{\A}\colon\mcg_\A(N)\to\mcg(N_\A)\] called \emph{reduction homomorphism along $\A$.} By \cite[Proposition 23]{ParisWBLN} we have
 \begin{equation}\label{eq:ker_cut}
  \ker\cut{\A}=\langle t_a\mid a\in\A\cap\curv^+(N)\rangle.   
 \end{equation}
%\item $\pAmcg{\A}(N)=\cut{\A}^{-1}(\pmcg(N_\A))$ -  elements of  $\pAmcg{\A}(N)$ preserve each curve of %$\A$ and each subsurface in $\sub{\A}(N)$, but may reverse orientation of some curves of $\A$. They also %preserve $\partial N$ componentwise.
%\item $\ppAmcg{\A}(N)=\cut{\A}^{-1}(\ppmcg(N_\A))$ -  elements of  $\ppAmcg{\A}(N)$ preserve each curve of %$\A$, preserve its orientation, and orientation of  its regular neighbourhood, and preserve  each %subsurface in $\sub{\A}(N)$. They also preserve each component of $\partial N$ and its orientation.

\subsection{Canonical reduction systems} 
We refer the reader to \cite{ParisWBLN} for a review, in the context of nonorientable surfaces, of the Nielsen-Thurston classification of mapping classes into three types: periodic, reducible and pseudo-Anosov.
Every $f\in\mcg(N)$ has an associated multicurve $\C(f)\subset\curv(N)$ called \emph{the canonical reduction system}, which is minimal among all multicurves $\A$ satisfying the property that $f(\A)=\A$ and for $m\ne 0$ such that $f^m$ fixes every component of $N_\A$, the restriction of $\cut{\A}(f^m)$ to the mapping class group of each component of $N_\A$ is either periodic or pseudo-Anosov \cite{BLM,WuY}.

Let $p\colon S\to N$ be the orientation double covering. It is proved in \cite{WuY} that for $f\in\mcg(N)$ we have
$\C(f)=p(\C(\widetilde{f}))$, where $\widetilde{f}\in\mcg(S)$ is the lift of $f$. This fact can be used to deduce the following basic facts about $\C(f)$ from the corresponding properties of $\C(\tilde{f})$ (see \cite[Proposition 2.2.7]{Castel}).
\begin{lemma}\label{lem:crs}
Let $f,g\in\mcg(N)$.
    \begin{enumerate}
        \item $\C(fgf^{-1})=f(\C(g))$.
        \item If $fg=gf$, then for any $a\in\C(f)$ and $b\in\C(g)$, $a\ne b$, we have $I(a,b)=0$.
        \item $\C(f)=\emptyset$ if and only if $f$ is either periodic or pseudo-Anosov.
    \end{enumerate}
\end{lemma}
We denote by $\C^+(f)$ the set of two-sided components of $\C(f)$
\[\C^+(f)=\C(f)\cap \curv^+(N).\]
Observe that (1) and (2) of Lemma \ref{lem:crs} are still true if $\C$ is replaced by $\C^+$. On the other hand, it is possible that $\C^+(f)=\emptyset$ but $\C(f)\ne\emptyset$.

\subsection{Almost pseudo-Anosov mapping classes.}
\begin{defi}
If $f\in \mcg(N)$ is such that $\C^+(f)=\emptyset$ and $\C(f)\neq \emptyset$, then we say that $f$ is \emph{almost pseudo-Anosov}.
\end{defi}

\begin{example}
   Let $S$ be a surface (orientable of not) with boundary and let $f'\in\mcg(S)$ be pseudo-Anosov. Choose $k$ components of $\partial S$ and glue a M\"obius strip $M_i$, $i=1,\dots,k$ along each of these components  to obtain a nonorientable surface $N$. Denote the core curve of $M_i$ by $\mu_i$. By extending $f'$ over $N$ we obtain an almost pseudo-Anosov $f\in\mcg(N)$ such that
   $\C(f)=\{\mu_1,\dots,\mu_k\}.$
   It follows from the proof of the next proposition that every almost pseudo-Anosov mapping class is obtained by this construction.
\end{example}

\begin{prop}\label{apA:virtuallyCyclic}
Let $f\in \mcg(N)$ be almost pseudo-Anosov and $\A=\C(f)$. Then
\begin{enumerate}
\item $\cut{\A}(f)$ is  pseudo-Anosov on $N_{\A}$;
\item the centralizer $C_{\mcg(N)}(f)$ of $f$ in $\mcg(N)$ is virtually cyclic.
\end{enumerate}
\end{prop}
\begin{proof}
Since $\A$ consists of one-sided curves, $N_{\A}$ is connected and \[\map{\cut{\A}}{\mcg_{\A}(N)}{\mcg(N_{\A})}\] is injective by \eqref{eq:ker_cut}. 
Let $f'=\cut{\A}(f)\in\mcg(N_{\A})$. By the definition of almost pseudo-Anosov, $\C(f)\neq \emptyset$ and $\C(f')=\emptyset$. In particular, $f$ is not periodic and by the injectivity of $\cut{\A}$ neither is $f'$. Hence, $f'$ is pseudo-Anosov.

If $g\in \mcg(N)$ and $gf=fg$, then by (1) of Lemma \ref{lem:crs}
\[g(\A)=g(\C(f))=\C(gfg^{-1})=\C(f)=\A.\]
Hence, $g\in \mcg_{\A}(N)$ and
\[C_{\mcg(N)}(f)\subset C_{\mcg_{\A}(N)}(f).\]
We have
\[\cut{\A}(C_{\mcg(N)}(f))\subset C_{\mcg(N_{\A})}(f'),\]
and because $\cut{\A}$ is injective, $C_{\mcg(N)}(f)$ is isomorphic to a subgroup of  $C_{\mcg(N_{\A})}(f')$, which is virtually cyclic by \cite[Poposition 24]{ParisWBLN}. It follows that $C_{\mcg(N)}(f)$ is also virtually cyclic.
\end{proof}

\begin{lemma}\label{apA:nofix}
If $f\in\mcg(N)$ is pseudo-Anosov or almost pseudo-Anosov, then $f(c)\ne c$ for every $c\in\curv^+(N)$.  
\end{lemma}
\begin{proof}
    A pseudo-Anosov mapping class does not fix any nontrivial curve, so we assume that $f$ is almost pseudo-Anosov and $\A=\C(f)$. Suppose that $f(c)=c$ for some $c\in\curv^+(N)$. Then $I(c,a)=0$ for every $a\in\A$ (see \cite{Kuno}). Since $c\notin\A$, we see that $\cut{\A}(c)$ is a nontrivial curve in $N_\A$ fixed by $\cut{\A}(f)$. This is a contradiction, because  $\cut{\A}(f)$ is pseudo-Anosov by (1) of Proposition \ref{apA:virtuallyCyclic}.
\end{proof}
%%%%%%%%%%%%%%%%%%%%%%%%%%%%%%%%%%%%%%%
%%%%%%%%%%%%%%%%%%%%%%%%%%%%%%%%%%%%%%
\section{Totally symmetric sets and multicurves}\label{sec:total_sym}
In this section we review the notions of a totally symmetric set introduced by Kordek and Margalit \cite{KM}, and totally symmetric multicurve introduced by Chen and Mukherjea \cite{CM}. We refer the reader to \cite{CM} for more background. 
\subsection{Totally symmetric sets.}
An ordered finite subset $X$ of a group $G$ is \emph{totally symmetric} if
\begin{enumerate}
    \item the elements of $X$ pairwise commute, and
    \item every permutation of $X$ can be obtained by conjugation by some element of $G$.
\end{enumerate}
The main example of a totally symmetric subset of $\B_n$ is the set of odd-index generators, which we denote by $\X_n$.
\[\X_n=\{\sigma_i\colon i\in\odd(n)\},\]
where $\odd (n)$ is the set of odd positive integers smaller that $n$. We will also use the following totally symmetric subset of the commutator subgroup $\B_n'$.
      \[\X'_n=\{\sigma_i\sigma_1^{-1}\colon i\in\odd(n), i\geq 3\}.\]
We have $|\X_n|=\lfloor\frac{n}{2}\rfloor$ and $|\X'_n|=\lfloor\frac{n}{2}\rfloor-1$. That $\X_n$ (respectively $\X'_n$) is a totally symmetric subset of $\B_n$ (respectively $\B_n'$) follows from the next useful lemma.
\begin{lemma}[Fact 2.2 in \cite{CM}]\label{lem:X_n_ts}
    Every permutation of $\X_n$ can be obtained by conjugation by some element of $\B_n'$.
\end{lemma}

\subsection{Totally symmetric multicurves.} 
Let $X$ be a set. A multicurve $\A$ is \emph{$X$-labelled} if each component $a$ of $\A$ is associated with a non-empty subset of $X$ called the \emph{label} of $a$ and denoted by $l(a)$. If $l(a)=X$ then we say that $a$ has \emph{full label}. We denote by $\sym(X)$ the symmetric group of $X$.
\begin{defi}
An $X$-labelled two-sided multicurve $\A$ on $N=N_{g,b}$ is \emph{totally symmetric} if for every $\tau\in\sym(X)$, there is $f_\tau\in\pmcg_\A(N)$ such that $l(f_\tau(a))=\tau(l(a))$ for every $a\in\A$.
\end{defi}
If $X$ is a totally symmetric subset of $\B_n$ and $\rho\colon\B_n\to\pmcg(N)$ is a homomorphism, then 
 \[\C^+(\rho(X))=\bigcup_{x\in X}\C^+(\rho(x))\]
    is a totally symmetric  $X$-labelled multicurve, where the label of a component $c$ is
    \[l(c)=\{x\in X\colon c\in\C^+(\rho(x))\}.\]
\begin{rem}
    We emphasise that in this paper all totally symmetric multicurves are two-sided by definition. In particular, the multicurve $\C(\rho(X))$ defined to be the union of $\C(\rho(x))$ for $x\in X$ is not totally symmetric, unless it coincides with $\C^+(\rho(X))$.
\end{rem}
\subsection{The squeeze map and $\sq$-faithful multicurves.}
By squeezing each boundary component of $N_{g,b}$ to a point we obtain a closed surface $N_g$ and a surjective continuous map (notation borrowed from \cite{Castel})
\[\sq\colon N_{g,b}\to N_{g}.\]
 % (c.f Definition 4.2.4 of Castel). Equivalently, we may glue a disc to each boundary component.

\begin{defi}
    A two-sided multicurve $\A=\{c_1,\cdots, c_r\}$ in $N_{g,b}$ is \emph{$sq$-faithful} if the curves  $\sq(c_1),\cdots, \sq(c_r)$ are nontrivial and pairwise nonisotopic in $N_g$.
\end{defi}
If $\A$ is an $\sq$-faithful totally symmetric $X$-labelled multicurve in $N_{g,b}$, then $\sq(\A)$ is a totally symmetric $X$-labelled multicurve in $N_{g}$ with labels $l(\sq(c))=l(c)$.

The following lemma can be proved by the same arguments as assertions (2) and (3) of \cite[Lemma 10.10]{CM} and we omit its proof.
\begin{lemma}\label{lem:forgetful_curve}
Let $\A$ be totally symmetric $X$-labelled multicurve in $N_{g,b}$ with $|X|\ge 3$ and let $c$ be a component of $\A$.
If $\sq(c)$ is trivial or $\sq(c)=\sq(c')$ for some $c'\in\A$, $c\ne c'$, then $l(c)=X$.
\end{lemma}
%\begin{proof}
%If $\sq(c)$ is trivial, then $c$ bounds either a punctured disc or a punctured M\"obius band.  If $c$ is not a type C curve, then %$|l(c)|\ne |X|$. By total symmetry, there is an element $h\in\pmcg(N)$ such that $h(c)$ is a component of $M$ and 
%$l(h(c))\ne l(c)$. Then $h(c)$ and $c$ are disjoint. This is not possible because  $h$ fixes every puncture.
%
%If $\sq(c)=\sq(c')$ then $N-c-c'$ has a component that is a punctured annulus $\mathcal{A}(c,c')$.  If $c$ is not a type C curve, then %$|l(c)|\ne |X|$. By total symmetry, there are elements $h_1,h_2\in\pmcg(N)$ such that $h_i(c)$ are components of $M$ for $i=1,2$ and  
%$l(h_1(c))$, $l(h_2(c))$ and $l(c)$ are different. Then $h_1(c)$, $h_2(c)$ and $c$ are pairwise disjoint. Then either $h_i(\mathcal{A}%(c,c'))$ is disjoint from $\mathcal{A}(c,c')$ or 
%$h_i(\mathcal{A}(c,c'))=\mathcal{A}(c,c')$. Since $h_i$ fix every puncture, we have $h_i(\mathcal{A}(c,c'))=\mathcal{A}(c,c')$ for $i=1,2$. %This means
%$h_1(c)=h_2(c)=c'$, which is a contradiction.
%\end{proof}

\begin{cor}\label{cor:sq_faithful}
If $\A$ is a totally symmetric $X$-labelled multicurve in $N_{g,b}$ with $|X|\ge 3$ such that every component of $\A$ has non-full label, then $\A$ is $\sq$-faithful.
\end{cor}

 \begin{cor}\label{cor:notCbound}
 If $\A$ is a totally symmetric $X$-labelled multicurve in $N_{g,b}$ with $|X|\ge 3$, then at most $\frac{3}{2}(g-2)$ components of $\A$ have non-full label.
 \end{cor}
 \begin{proof}
 Let $c_1,\dots,c_r$ be all the components of $\A$ that have non-full label. By Lemma \ref{lem:forgetful_curve}, $\sq(c_1),\dots,\sq(c_r)$ are nontrivial and pairwise nonisotopic. It means that they form a two-sided multicurve in the closed surface $N_g$ and therefore $r\le\frac{3}{2}(g-2)$ by the standard argument using the Euler characteristic. 
 \end{proof}

 \subsection{Special, normal and exotic components.}
 \begin{defi}
Let $\A$ be a totally symmetric $X$-labeled multicurve in $N_{g,b}$ with $|X|=k$
 and let $a$ be a component of $\A$. We say that $a$ is
 \begin{enumerate}
     \item \emph{special} if $|l(a)|=1$,
     \item \emph{normal} if $|l(a)|=k$,
     \item \emph{exotic} if $|l(a)|=k-1$.
 \end{enumerate}
 \end{defi}
 The above definition is analogous to one in \cite[Sec. 3]{CM}, but our terminology is different. The terms ``special'' and ``normal'' are borrowed from \cite{Castel}.  
 The following proposition is analogous to \cite[Proposition 3.1]{CM}, but we don't assume that the surface is closed.
 \begin{prop}\label{ch3:curves:AIC}
    If $\A$ is totally symmetric $X$-labeled multicurve in $N_{g,b}$ with $|X|=k$ and $k^2-k>3(g-2)$, then every component of $\A$ is either special or normal or exotic. In particular, this statement holds when $k\geq 7$ and $g\leq 2k+1$.
\end{prop}
\begin{proof}
    If $c$ is a component of $\A$ that is not normal and $|l(c)|=l$, then $\A$ has at least $\binom{k}{l}$ components with non-full label. By Corollary \ref{cor:notCbound}, we have  \[\binom{k}{l}\le\frac{3}{2}(g-2).\]
    If  $k-2\geq l\geq 2$ then 
        \[\frac{k(k-1)}{2}\leq \binom{k}{l}\leq \frac{3}{2}(g-2),\]
 which contradicts the assumption $k(k-1)> 3(g-2)$. Thus $l\in\{1,k-1\}$, which means that $c$ is either special or exotic.

   It is easy to check that for $k\geq 7$ and $g\leq 2k+1$ we have $k^2-k>3(g-2)$.
%If $k\geq 7$ and then \[k^2-k\geq 7k-k=6k\ge3(n-1)\geq 3(g-2).\]
%    \[\begin{aligned}
%      &3(g-2)\leq 3(n+1-2)\leq 6k-3  \\
%      &
%    \end{aligned}\]
\end{proof}

\subsection{A remark about the paper of Chen and Mukherjea.\label{CMbug}}
One difference between our paper and \cite{CM} is that we allow the surface to have boundary throughout the whole proof, whereas Chen and Mukherjea first work with a closed surface. Then they use the natural projection 
\[\sqa\colon\pmcg(S_{g,b})\to\mcg(S_g)\]
induced by the squeeze map $\sq\colon S_{g,b}\to S_g$
 to deduce information about $\C(f)$ for \(f\in\pmcg(S_{g,b})\) by looking at $\C(\sqa(f))$. 
 In Section 10.3 of \cite{CM} the authors claim that $\C(\sqa(f))$ can be obtained from $\sq(\C(f))$ by identifying isotopic curves and deleting trivial curves. Unfortunately, this claim is false in general, as the following examples demonstrate. By pointing out this issue we don't question correctness of the result of Chen and Mukherjea, only accuracy of their argument. In fact, our paper shows that the methods of \cite{CM}, such as their proof of Proposition 3.1 corresponding to our Proposition \ref{ch3:curves:AIC}, can be modified to work also for surfaces with boundary.
\begin{example}
 Let $\A=\{a,c,d\}$ be a multicurve in $S_{g,b}$, $b>0$, such that $\sq(a)\ne\sq(c)=\sq(d)$, and let $f=t_at_ct_d^{-1}$. Then
 $\C(f)=\A$ and  $\sqa(f)=t_{\sq(a)}$ so 
 \[\C(\sqa(f))=\{\sq(a)\}\ne\{\sq(a),\sq(c)\}=\sq(\C(f)).\]
\end{example}

\begin{example}
 By Penner's construction (see Theorem 14.4 in \cite{FM}) we can produce a pseudo-Anosov $f\in\pmcg(S_{g,b})$ such that $\sqa(f)$ is reducible.  
If multicurves $a=\{a_1,a_2,a_3\}$ and $b=\{b_1,b_2\}$ are as in Figure \ref{fig:pa:ex}, then $f=t_at_b^{-1}\in \pmcg(S_{2,2})$ is pseudo-Anosov. However, 
\[\sqa(f)=t_{\sq(a_2)}^2\]
is reducible.
 \begin{figure}[h]
\begin{center}
\includegraphics[width=0.8\customwidth]{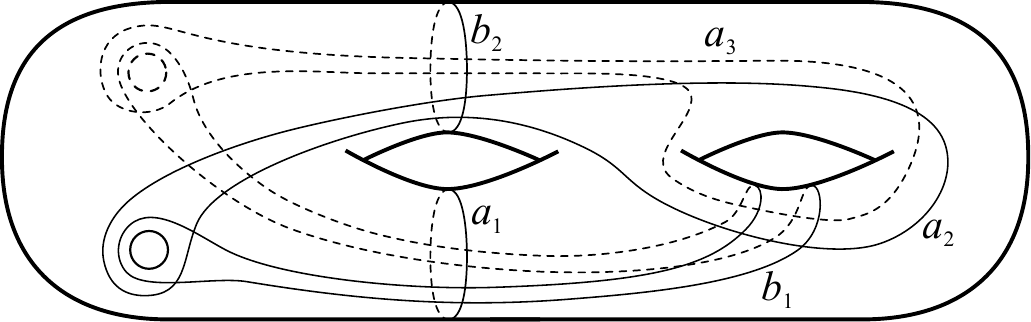}
\caption{Two multicurves $a=\{a_1,a_2,a_3\}$ and $b=\{b_1,b_2\}$ filling $S_{2,2}$.}\label{fig:pa:ex} %
\end{center}
\end{figure}
\end{example}

%%%%%%%%%%%%%%%%%%%%%%%%%%%%%%%%%%%%%%%
%%%%%%%%%%%%%%%%%%%%%%%%%%%%%%%%%%%%%%
\section{Proof of Theorem \ref{thm:main_pmcg}}\label{sec:start}
Throughout the rest of this paper we assume that $n\ge 14$, $b\ge 0$, $g\leq2\lfloor{n}/{2}\rfloor+1$,  and 
$\rho\colon\B_n\to\pmcg(N_{g,b})$ is a homomorphism. 
We consider the totally symmetric $\X_n$-labelled multicurve
 \[\C^+(\rho(\X_n))=\bigcup_{\sigma_i\in\X_n}\C^+(\rho(\sigma_i)).\]
 By Proposition \ref{ch3:curves:AIC} every component of $\C^+(\rho(X_n))$ is either special, or normal or exotic.
 Here is the outline of the proof of Theorem \ref{thm:main_pmcg}.
 \begin{enumerate}
     \item In Section \ref{sec:irred} we consider the case where $\C^+(\rho(X_n))$ is empty, that is where $\rho(\sigma_i)$ is either periodic, or pseudo-Anosov or almost pseudo-Anosov. We prove that $\rho$ is cyclic in this case.
    
    \item In Section \ref{sec:cyclic_action} we prove that if all components of $\C^+(\rho(X_n))$ are normal, then $\rho$ is cyclic. 
    
    \item In Section \ref{sec:no_I_cirves} we prove that $\C^+(\rho(X_n))$ has no exotic components. 
    
    \item In Section \ref{sec:top_A} we assume that $\C^+(\rho(X_n))$ has special components and we denote by $A_i$ the set of components with label $\{\sigma_i\}$. We prove that one of the following two cases occurs.
    \begin{description}
        \item[Case 1] $A_i$ contains a unique separating curve $a_i$ for each $i$, or
        \item[Case 2] $A_i$ consists of a single nonseparating curve $a_i$ for each $i$.
    \end{description}

    \item In Section \ref{sec:delta} %we introduce  the element $\delta=\sigma_1\sigma_2\cdots\sigma_{n-1}$ and the $n$-th generator $\sigma_0=\delta\sigma_{n-1}\sigma^{-1}$ of $\B_n$. 
    we extend the notion of special and normal curves over all generators $\sigma_i$, including even $i$.

    \item In Section \ref{sec:separating} we prove that in Case 1 
    $(a_1,\dots,a_{n-1})$ is a chain of separating curves.

    \item In Section \ref{sec:transvection} we construct a transvection $\rho^\tau$ of $\rho$, such that
    \begin{enumerate}
        \item $\rho^\tau(\sigma_i)$ is the identity outside the one-holed Klein bottle bounded by $a_i$ for $1\le i\le n-1$ in Case 1, or
        \item $\rho^\tau(\sigma_i)$ is a power of Dehn twist about  $a_i$ for $1\le i\le n-1$ in Case 2.
    \end{enumerate}

    \item In Section \ref{sec:finish} we prove that
    \begin{enumerate}
        \item $\rho^\tau$ is a {\ctr} in Case 1, or
        \item $\rho^\tau$ is a standard twist representation in Case 2.
    \end{enumerate}
\end{enumerate}
%%%%%%%%%%%%%%%%%%%%%%%%%%%%%%%%%%%%%%%
%%%%%%%%%%%%%%%%%%%%%%%%%%%%%%%%%%%%%%
\section{Periodic and (almost) pseudo-Anosov representations}\label{sec:irred}
In this section we prove Theorem \ref{thm:main_pmcg} in the case where $\C^+(\rho(\sigma_i))$ is empty, that is where
$\rho(\sigma_i)$ is either periodic, or pseudo-Anosov, or almost pseudo-Anosov.

\begin{lemma}\label{lem:cyclic_bdry}
For $b>0$, every finite subgroup of $\ppmcg(N_{g,b})$ is cyclic.
\end{lemma}
\begin{proof}
By \cite[Lemma A.3.11]{Castel} every finite subgroup of $\ppmcg(S_{g-1},2b)$ is cyclic. Since $\ppmcg(N_{g,b})$ is isomorphic to a subgroup of  $\ppmcg(S_{g-1,2b})$ by Corollary \ref{cor:pureBirChil}, the lemma follows. 
\end{proof}

%\begin{lemma}\label{lem:torsion}
%    If $G$ is a finite abelian subgroup of $\mcg(N_{g})$, then $|G|\leq 2g$ for $g\neq 6$ and $|G|\leq 16$ for $g=6$.
%\end{lemma}

The next lemma is analogous to \cite[Proposition 5.2]{CM}
\begin{lemma}\label{lem:periodic_closed}
If $\mcg(N_{g})$ contains a totally symmetric set consisting of $k\geq 6$ periodic elements, then $g\geq 2^{k-2}$.
\end{lemma}
\begin{proof}
    By \cite[Lemma 5.3]{CM}, $\mcg(N_{g})$ contains an abelian finite subgroup $G$ of order at least $2^{k-1}$. 
    By  Kerckhoff's Theorem for nonorientable surfaces \cite{CX,Kerck}, %(see Remark at the end of Section IV in \cite{Kerck}),
    $G$ can be realized as a group of automorphisms of a nonorientable unbordered Klein surface of genus $g$. 
    By \cite{Rod}, the maximum order of an abelian group of automorphisms of such a surface is $2g$ for $g\neq 6$ and $16$ for $g=6$.
    Since $|G|\geq 2^{k-1}\geq 32$, we have $|G|\leq 2g$ and $2^{k-2}\leq g$.
     % \[\begin{aligned}
     %    2^{k-1}&\leq 2g\\
     %   2^{k-2}&\leq g.
    %\end{aligned}\]
\end{proof}

\begin{prop}\label{ch:5:corollary5:5}
    %Let $n\geq 14$, $g\leq n+1$, $b\ge 0$ and let $\rho\colon\B_n\to\pmcg(N_{g,b})$ be a homomorphism. 
    If $\rho(\sigma_1\sigma_3^{-1})$ is periodic, then $\rho$ is cyclic. 
\end{prop}
\begin{proof}% We split the proof in two cases.
{\bf Case 1: $b>0$}. In this case we follow the proof of \cite[Proposition 5.2.1]{Castel}.
Since \[\pmcg(N_{g,b})/\ppmcg(N_{g,b})\cong \Z_2^b\] is abelian, $\rho(\B_n')$ is a subgroup of $\ppmcg(N_{g,b})$. 
Suppose that $\rho(\sigma_1\sigma_3^{-1})$ has order $m$. 
Then $\rho(\sigma_1\sigma_5^{-1})$ also has order $m$ and commutes with $\rho(\sigma_1\sigma_3^{-1})$. Let $G$ be the subgroup of  $\ppmcg(N_{g,b})$ generated by $\rho(\sigma_1\sigma_3^{-1})$ and $\rho(\sigma_1\sigma_5^{-1})$ and note that it is a finite abelian group, whose every element  has order at most $m$. But $G$ is cyclic by Lemma \ref{lem:cyclic_bdry}, so it must have order $m$. It follows that each of $\rho(\sigma_1\sigma_3^{-1})$ and $\rho(\sigma_1\sigma_5^{-1})$ individually generates $G$. In particular, $\rho(\sigma_1\sigma_5^{-1})$ belongs to the subgroup generated by $\rho(\sigma_1\sigma_3^{-1})$. By the same argument,
$\rho(\sigma_1\sigma_6^{-1})$ belongs to the subgroup generated by $\rho(\sigma_1\sigma_3^{-1})$. It follows that $\rho(\sigma_1\sigma_6^{-1})$ commutes with $\rho(\sigma_1\sigma_5^{-1})$. But these elements satisfy the braid relation, so 
$\rho(\sigma_1\sigma_5^{-1})=\rho(\sigma_1\sigma_6^{-1})$ and $\rho(\sigma_5)=\rho(\sigma_6)$. It follows that $\rho$ is cyclic.

\medskip
{\bf Case 2: $b=0$}. In this case we follow the proof of \cite[Corollary 5.5 ]{CM}.
By \cite[Lemma 2.3]{CM}, $\rho(\X'_n)$ is either a singleton, or a totally symmetric set consisting of  $\lfloor\frac{n}{2}\rfloor-1$  periodic elements. The latter is impossible by Lemma \ref{lem:periodic_closed} since for $n\geq 14$  and $k=\lfloor{n}/{2}\rfloor-1$ we have
    \[g\leq 2(k+1)+1=2k+3<2^{k-2}.\]
So $\rho(\sigma_3\sigma_1^{-1})=\rho(\sigma_5\sigma_1^{-1})$ and thus $\sigma_3\sigma_5^{-1}\in\ker\rho$. Since $\sigma_3\sigma_5^{-1}$ normally generates $\B_n'$ it follows that $\rho$ is cyclic.     
\end{proof}

Observe that Proposition \ref{ch:5:corollary5:5} implies Theorem \ref{thm:main_pmcg} in the case where $\rho(\sigma_1)$ is periodic. Indeed, if  $\rho(\sigma_1)$ is periodic then so is $\rho(\sigma_1\sigma_3^{-1})$, because $\sigma_1$ and $\sigma_3$ are conjugate and they commute. 
In the case where $\rho(\sigma_1)$ is pseudo-Anosov or almost pseudo-Anosov Theorem \ref{thm:main_pmcg} follows from the next proposition.

\begin{prop}
     %Let $n\geq 14$, $g\leq n+1$, $b\ge 0$ and let $\rho\colon\B_n\to\pmcg(N_{g,b})$ be a homomorphism.  
     If $\rho(\sigma_1)$ is pseudo-Anosov or almost pseudo-Anosov, then $\rho$ is cyclic.
\end{prop}
\begin{proof}
    The proof is essentially the same as in \cite[Proposition 6.1]{CM}, so we omit the details. The key fact is that the centralizer of a  pseudo-Anosov (respectively almost pseudo-Anosov) element is virtually cyclic by \cite[Poposition 24]{ParisWBLN} (respectively Lemma \ref{apA:virtuallyCyclic}). From this fact we deduce that $\rho(\sigma_1\sigma_3^{-1})$ is periodic (see \cite{CM} for details) and
    $\rho$ is cyclic by Proposition \ref{ch:5:corollary5:5}.
\end{proof}
%%%%%%%%%%%%%%%%%%%%%%%%%%%%%%%%%%%%%%
%%%%%%%%%%%%%%%%%%%%%%%%%%%%%%%%%%%%%%%
\section{Multicurves preserved by the action of $\rho(\B_n)$}\label{sec:cyclic_action}
The main goal of this section is to prove Theorem \ref{thm:main_pmcg} in the case when all 
curves in the labeled multicurve $\C^+(\rho(\X_n))$ are normal -- see Proposition~\ref{all_C_case} below. This case is a bit tricky, because being \emph{totally symmetric} is rather meaningless in the case of normal curves. On the other hand, normal curves constitute a multicurve which is preserved by the action of $\B_n$, and this is the starting point for our analysis -- see Proposition~\ref{prop:ch8:8.1} below.
\begin{lemma}\label{lem:ch8:max:nonseparating}
    If $\A$ is a two-sided multicurve in $N=N_{g,b}$ such that $N_\A$ is connected, then $|\A|\leq \lfloor{g}/{2}\rfloor$. In particular $|\A|\leq g-2$ for $g\geq 4$.
\end{lemma}
\begin{proof}
By the assumption, $N_\A$ is a connected surface with $2|\A|+b$ boundary components. Hence
    \[2-2|\A|-b\geq \chi\left(N_\A\right)=\chi\left(N\right)=2-g-b.\]
    This proves that $|\A|\leq \frac{g}{2}$. If $g\geq 4$, then $\frac{g}{2}\leq g-2$.
    %\[|\A|\leq \frac{g}{2}=\frac{2g-g}{2}<\frac{2g-4}{2}=g-2.\]
\end{proof}
%\begin{rem}
%    If $g=3$ or $g=4$, then it is easy to find $g-2$ disjoint, nonispotopic, two-sided scc which do not separate $N_g$. 
%\end{rem}
\begin{lemma}\label{lem:ch8:max:homologous}
    Let $g\geq 4$ and $\A$ be an $\sq$-faithful two-sided multicurve in $N_{g,b}$ such that the curves in $\sq(\A$) are pairwise homologous in $H_1(N_g,\Z_2)$ and of the same topological type. Then $|\A|\leq g-2$.
\end{lemma}
\begin{proof}
Let $\sq(\A)=\{c_1,\ldots,c_r\}$ and assume first that $[c_1]=0$ in $H_1(N_g,\Z_2)$. Hence, $c_1$ is a separating curve in $N_g$ and it bounds either a nonorientable subsurface of  genus at least 2, or an orientable subsurface of  genus at least 1. We assumed that all elements of $\A$ are of the same topological type, so the same statement is true for every $c_i$ and we have
\[r\leq \frac{g}{2}=\frac{2g-g}{2}\leq \frac{2g-4}{2}=g-2,\]
    as long as $g\geq 4$.
    
    If $[c_1]\neq 0$ in $H_1(N_g,\Z_2)$, then each of $c_i$ is nonseparating in $N=N_g$ and each component of $N_{\sq(\A)}$ has exactly two boundary components %marked with elements of $\A$ 
    (because elements of $\sq(\A)$ are pairwise homologous). In particular, $N_{\sq(\A)}$ has exactly $r$ components and the genus of each of these components is at least 1. Hence, 
    \[2-g=\chi(N)=\chi(N_{\sq(\A)})\leq -r,\]
 which gives $r\leq g-2$.
\end{proof}
%\begin{rem}
%    It is easy to construct a set of $g-2$ disjoint, nonseparating and pairwise homologous scc in $N_g$: glue together $g-2$ projective planes with two holes.
%\end{rem}
\begin{lemma}%[Claim 8.3 of CM] 
\label{prop:ch8:Claim} Assume that $g>4$ and $|X|\geq 3$. If $\A$ is a totally symmetric $X$-labelled multicurve in $N_{g,b}$ such that $|l(c)|=1$ for every $c\in \A$ and $\map{l}{\A}{2^X}$ is injective, then $|\A|\leq g-2$.
\end{lemma}
\begin{proof}
By definition, $\A$ is two-sided and by Corollary \ref{cor:sq_faithful}, $\A$ is $\sq$-faithful. Suppose that 
$\sq(\A)=\{c_1,c_2,\ldots, c_r\}$ and $r\geq g-1$. 

Lemma \ref{lem:ch8:max:nonseparating} implies that some subset of $\sq(\A)$, say $\{c_1,\ldots, c_{s}\}$ where $s\leq g-2$,  bounds a subsurface in $N_g$. In particular, in $H_1(N_g,\Z_2)$ we have:
\[[c_1]+[c_2]+\ldots+[c_s]=0.\]
Since $\sq(\A)$ is totally symmetric, there exists an element of $\mcg(N_g)$ which transposes $c_1$ with $c_{s+1}$ and fixes all other components of $\sq(\A)$. Therefore,
\[[c_{s+1}]+[c_2]+\ldots+[c_s]=0\]
and $[c_{s+1}]=[c_1]$. This implies that components of $\sq(\A)$ are pairwise homologous. They are also of the same topological type, hence Lemma~\ref{lem:ch8:max:homologous} implies that $r\leq g-2$. This contradicts our assumption that $r\geq g-1$.
\end{proof}
%\begin{rem}
%The cases of $g\leq 4$ would need a special consideration. The statement may be true if %$g=3$ or $g=4$, but we probably won't need this.
%\end{rem}
%\begin{rem}
%The following proposition is quite fundamental for both CM and Castel. I think Castel's proof works more or less ok, so the lemma is valid under lesser assumptions, for example Castel do not assume that the surface is closed.
%\end{rem}
%\begin{rem}
%    OK, so we now (17.01.2024) upgrade the statement to punctured surfaces. The main problem here is that both Castel and CM work here without 0-peripheral curves (for different reasons -- Castel removed them in chapter 4, and in CM the surface is closed). The final argument is a mix of elements from both sources + the notion of $\sq$-faithful curves.
%\end{rem}
% Lemma 8.2 of CM and Proposition 7.1.3 of Castel
\begin{defi}
We say that a component of a two-sided multicurve $\A$ is \emph{0-peripheral} if it separates a genus $0$ subsurface.
\end{defi}
\begin{lemma}\label{lem:0per}
Let $\A$ be a multicurve in  $N=N_{g,b}$ without 0-peripheral components. Every component $R$ of $N_\A$ satisfies
\[|\bndy(R)\cap\A|\le g.\] 
\end{lemma}
\begin{proof}
Let $\overline{N}=\sq(N)$. Since $\A$ has no 0-peripheral components, %$\sq(\A)$ has no trivial components. It follows that 
every component of $\overline{N}_{\sq(\A)}$ has a non-positive Euler characteristic. Note that if $m$ is the number of boundary components of $\sq(R)$, then $m=|\bndy(R)\cap\A|$ and 
\[2-g=\chi\left(\overline{N}\right)=\chi\left(\overline{N}_{\sq(\A)}\right)\leq\chi\left(\sq(R)\right)\le 2-m.\]
%. Let $R$ be a component of 
%$\overline{N}_{\sq(\A)}$ and let $m$ be the number of its boundary components. We have
Hence $m\le g$.
\end{proof}
\begin{prop}\label{prop:ch8:8.1}
Assume that $n\geq 6$ and $g\leq n+1$. Let $\map{\rho}{\B_n}{\pmcg(N_{g,b})}$ be a homomorphism such that $\rho(\B_n)$ preserves the isotopy class of some %$\sq$-faithful} 
two-sided multicurve $\A$. Then 
\begin{enumerate}
    \item the induced action of $\B_n$ on the set of components of $N_\A$ is cyclic;
    \item the induced action of $\B_n$ on the set of components of $\A$ is cyclic;
    \item elements of $\rho(\B_n')$ preserve sides and orientation of every component of~$\A$.
\end{enumerate}
\end{prop}
\begin{proof}
Observe that $\B_n$ fixes every 0-peripheral component of $\A$, and also fixes the genus $0$ subsurface seprated by such a component.
Therefore we can assume that $\A$ has no 0-peripheral components.
Let 
\[N_\A=\bigcup_{i=1}^r R_i\cup \bigcup_{i={r+1}}^s R_i,\]
where each $R_i$, $1\leq i\leq r$, is a connected subsurface disjoint from the external boundary components of $N_\A$, and each $R_i$, $i>r$ is a component of $N_\A$ which contains at least one external boundary component. Since $\rho(\B_n)$ preserves $\A$, it permutes the set $\{R_i\}_{i=1}^{s}$. %Note also that $\A$ is $\sq$-faithful, so all components of $\A$ have negative Euler characteristic.

(1) Denote by $b_i$ the number of external boundary components of $R_i$ for $r+1\leq i\leq s$. Since $\A$ has no 0-peripheral components, 
if $R_i$ has genus $0$ then it has at least two non-external boundary components and $\chi(R_i)\le -b_i$. Moreover, $\chi(R_i)\leq -1$ for $1\leq i\leq r$, so 
\[\begin{aligned}
2-g-b&=\chi(N)=\chi(N_\A)\leq r\cdot (-1) -\sum_{i=r+1}^s b_i=-r-b,\\
          r&\leq g-2\leq(n+1)-2=n-1.
        \end{aligned}\]
By Artin's theorem \cite{Artin},  the action of $\B_n$ on $\{R_i\}_{i=1}^{r}$ is cyclic. 
Furthermore, elements of $\pmcg(N_{g,b})$ preserve the external boundary components, so $\rho(\B_n)$ fixes every component of $\bigcup_{i={r+1}}^s R_i$. Therefore, the action of $\B_n$ on $\{R_i\}_{i=1}^{s}$ is cyclic.

(2) By (1), the induced action of $\B_n'$ on elements of $\{R_i\}_{i=1}^{s}$ is trivial. Hence, there is a well defined induced action of $\B_n'$ on $\bndy(R_i)\cap \A$. Since all abelian quotients of $\B_n$ are cyclic, it is enough to prove that $\B_n'$ acts trivially on $\A$. Even more, it is enough to prove that $\B_n'$ acts trivially on $\A_i=\bndy(R_i)\cap \A$ for any $1\leq i\leq s$.

%\red{By Lemma \ref{lem:0per} for any $1\leq i\leq s$ we have \[|\bndy(R_i)\cap \A|\le g\le n+1.\]}

If $|\A_i|<n$, then by Theorem A of \cite{Lin}, the action of $\B_n'$ on $\A_i$ is trivial, so from now on assume that $|\A_i|\geq n$. 

If $\rho(\B_n)$ does not fix $R_i$, then there is another component $R_j$ of $N_\A$ homeomorphic to $R_i$. Moreover,  $\rho(\B_n)\subset \pmcg(N)$, so $i,j\leq r$. Then
\[\chi(R_j)=\chi(R_i)\leq 2-|\bndy(R_i)|\leq 2-n.\]
Since $g\leq n+1$, we have
\[\begin{aligned}
2\chi(R_i)-b&\geq \chi(N_\A)=\chi(N),\\
    4-2n-b&\geq 2-g-b\geq 2-(n+1)-b,\\
    3&\geq n.
\end{aligned}\]
This contradicts our assumption that $n\geq 6$. Hence, $\rho(\B_n)$ fixes $R_i$ and there is an induced action of $\B_n$ on $\A_i$. We can also assume that this action is transitive -- otherwise we can restrict the action to its orbits and use the same arguments.

If $|\A_i|=n$, then Artin's theorem \cite{Artin} implies that we have a standard action of $\B_n$ on $\A_i$. If we add $n$ different labels to components of $\A_i$, we get a totally symmetric multicurve in $N$ which satisfies the assumptions of Lemma \ref{prop:ch8:Claim}. Hence,
\[n=|\A_i|\leq g-2\leq n-1,\]
which is not possible. 

%Finally, assume that $|\partial R_i\cap \A|>n$. As we observed in the proof of Corollary, \ref{cor:notCbound}, $|\A|\leq \frac{3}{2}(g-%2)$. Therefore,
%\[|\partial R_i\cap \A|\leq |\A|\leq \frac{3}{2}(g-2)\leq 
%\frac{3}{2}(n-1)<2n\]

Finally, if $|\A_i|>n$, then by Lemma \ref{lem:0per}, 
\[|\A_i|\le g\le n+1,\]
so $|\A_i|=n+1$ and we can use Theorem F of \cite{Lin} to conclude that the action of $\B_n$ on $\A_i$ is cyclic. 
 
 (3) Let $a$ be a component of $\A$, $G$  the stabiliser of $a$ in $\pmcg(N)$ and $H$ the subgroup of $G$ consisting of elements that preserve sides and orientation of $a$. Note that $H$ is a normal subgroup of $G$ of index at most $4$ and $G/H$ is abelian.
 By (2) we have $\rho(\B_n')\subset G$ and since  $\B_n'$ is perfect for $n\ge 5$ \cite{GorLin}, we have $\rho(\B_n')\subset H$.
 \end{proof}

\begin{rem}\label{rem:twist_comm}
Under the assumptions of Proposition \ref{prop:ch8:8.1}, for $f\in\rho(\B'_n)$ and $a\in\A$ we have $ft_af^{-1}=t_a$. Indeed, by (2) we have $f(a)=a$, and by (3) $f$ preserves orientation of a regular neighbourhood of $a$.
\end{rem}

The proof of the next proposition is very similar to the proof of Proposition 8.1 of \cite{CM}, but we do not assume that the surface $N$ is closed.

%\begin{rem}
%    The next proposition makes use of Corollary 5.5 of CM, and this is the only reason for the assumption about $n$ being large. 
%\end{rem}
\begin{prop}\label{all_C_case}
    %Let $n\geq 14$, $g\leq  2\cdot \nhalf+1$ and let $\map{\rho}{\B_n}{\pmcg(N_{g,b})}$ be a homomorphism.  %(as in Corollary 5.5 of CM) 
    If the labeled multicurve $\C^+(\rho(\X_n))$ contains only normal curves, then $\rho$ is cyclic.
\end{prop}
\begin{proof}
    Since all components of $\A=\C^+(\rho(\X_n))$ are normal, for each  $i\in\odd(n)$ we have $\C^+(\rho(\sigma_i))=\A$. Moreover, each $\sigma_j$ with even  $j$ commutes with some $\sigma_i$ with $i\in\odd(n)$, hence Lemma \ref{lem:crs} implies that $\rho(\B_n)$ preserves the isotopy class of $\A$. In particular,  elements of $\rho(\B_n)$ permute the connected components $\{R_i\}_{i=1}^r$ of $N_\A$. By Proposition \ref{prop:ch8:8.1}, the induced action 
    \[\map{\pi\circ \cut{\A}\circ \rho}{\B_n}{\sym(\{R_i\}_{i=1}^r)}\]
    is cyclic, so $\B_n'\subset K=\ker\left(\pi\circ \cut{\A}\circ \rho\right)$. If we restrict $\rho$ to $K$, 
    we obtain a homomorphism
    \[\map{\overline{\rho}}{K}{\pmcg(R_1)\times \pmcg(R_2)\times\ldots\times \pmcg(R_r)},\]
    which further restricts to a homomorphism
    \[\map{\overline{\rho}_i}{K}{\pmcg(R_i)},\quad \text{for $1\leq i\leq r$.}\]
    Note that $K$ is of finite index in $\B_n$, hence 
    $\sigma_1^p\in K$ for some $p\geq 1$. Moreover, $\sigma_1$ and $\sigma_3$ are conjugate by an element of $\B_n'\subset K$, hence $\sigma_3^p\in K$ and $\overline{\rho}_i(\sigma_3)^p$ is conjugate to $\overline{\rho}_i(\sigma_1)^p$ in $\pmcg(R_i)$. These two elements also commute, and \[\C^+\left(\overline{\rho}_i(\sigma_1)^p\right)=\C^+\left(\overline{\rho}_i(\sigma_3)^p\right)=\emptyset .\]
    Therefore, they are either elements of the same finite order, or they are both pseudo-Anosov, or else they are both almost pseudo-Anosov. By Proposition \ref{apA:virtuallyCyclic}, in every case, there is $q\geq 1$ such that 
    \[\begin{aligned}
      \left(\overline{\rho}_i(\sigma_1)^p\right)^q&=\left(\overline{\rho}_i(\sigma_1)^p\right)^q,\\
      \overline{\rho}_i(\sigma_1\sigma_3^{-1})^{pq}&=1.
    \end{aligned}\]
    This implies that $\overline{\rho}(\sigma_1\sigma_3^{-1})$ has finite order. If we denote its order by $m$, then
    \[{\rho}(\sigma_1\sigma_3^{-1})^m\in \ker \cut{\A}=
    \langle t_a\mid a\in\A\rangle\]
    (see Section \ref{sub:sec:multicurves} for the properties of $\cut{\A}$). %Recall that the action of $\B_n'$ on components of $\A$ is trivial, and 
    Since $\sigma_1\sigma_5^{-1}$ is conjugate to $\sigma_1\sigma_3^{-1}$ by an element of $\B_n'$, from Proposition \ref{prop:ch8:8.1} and Remark \ref{rem:twist_comm} we obtain
    \[\begin{aligned}
        \rho(\sigma_1\sigma_3^{-1})^m&=\rho(\sigma_1\sigma_5^{-1})^m,\\
        \rho(\sigma_3\sigma_5^{-1})^m&=1.
    \end{aligned}\]
    Thus, we proved that $\rho(\sigma_3\sigma_5^{-1})$ has finite order. The same is true for its conjugate $\rho(\sigma_1\sigma_3^{-1})$, so Proposition \ref{ch:5:corollary5:5} implies that $\rho$ is cyclic.
\end{proof}

%%%%%%%%%%%%%%%%%%%%%%%%%%%%%%%%%%%%%%
%%%%%%%%%%%%%%%%%%%%%%%%%%%%%%%%%%%%%%%
\section{Ruling out exotic curves}\label{sec:no_I_cirves}
The main goal of this section is to prove that the labelled multicurve 
$\C^+(\rho(\X_n))$ does not contain exotic curves -- see Proposition \ref{eliminate:exotic} below. The proof of this theorem follows the lines of the proof of Proposition 9.1 of \cite{CM}, however we want to correct one inaccuracy from the original proof (see Remark \ref{rem:9.1:expl} below), so we repeat the argument.
\begin{prop}\label{eliminate:exotic} %[Proposition 9.1 of CM] 
    %Let $n\geq 14$, $g\leq  2\cdot \nhalf+1$ and let $\map{\rho}{\B_n}{\pmcg(N_{g,b})}$ be a homomorphism. 
    The labelled multicurve $\C^+(\rho(\X_n))$ cannot contain exotic curves.
\end{prop}
\begin{proof}
    By Proposition \ref{ch3:curves:AIC}, every curve of $\C^+(\rho(\X_n))$ is special, normal or exotic, so we can decompose this multicurve as follows:
    \[\C^+(\rho(\X_n))=\bigcup_{i\in\odd(n)}A_i\cup \bigcup_{i\in\odd(n)}E_i\cup C,\]
    where $C$ is the set of normal curves, $A_i$ are special curves contained in $\C^+(\rho(\sigma_i))$, and $E_i$ are exotic curves which do not contain $i$ in their label set. Define also 
    \[\A=\bigcup_{i\in\odd(n)}E_i\cup C.\]
    We will show that there is an induced action of $\B_n$ on $\A$. 

Since all elements in $\X_n$ commute, for any $i,j\in\odd(n)$,  $\rho(\sigma_j)$ preserves $\C^+(\rho(\sigma_i))$. Hence, $\rho(\sigma_j)$ preserves
     \[\bigcap_{i\in\odd(n)} \C^+(\rho(\sigma_i))=C\]
     and
     \[\bigcap_{i\in\odd(n),\ i\neq j} \C^+(\rho(\sigma_i))=\C\cup E_j.\]
     This implies that each element of $\rho(\X_n)$ preserves $C$ and every $E_j$ for $j\in \odd(n)$. In particular, there is an action of  $\rho(\X_n)$ on $\A$.

     We now argue that if $1\leq j\leq n$ is even, then $\rho(\sigma_j)$ also preserves $\A$. In order to simplify the notation, we assume that $j=2$. Since $\sigma_2$ commutes with $\sigma_i$ for $i=5,7,9$, $\rho(\sigma_2)$ preserves the following sets
     \[\begin{aligned}
        &\C^+(\rho(\sigma_7))\setminus\C^+(\rho(\sigma_5))=A_7\cup E_5,\\
    &\C^+(\rho(\sigma_9))\setminus\C^+(\rho(\sigma_5))=A_9\cup E_5,\\
     &\left(A_7\cup E_5\right)\cap \left(A_9\cup E_5\right)=E_5,\\
          &\C^+(\rho(\sigma_5))\setminus\C^+(\rho(\sigma_7))=A_5\cup E_7,\\
         &\C^+(\rho(\sigma_9))\setminus\C^+(\rho(\sigma_7))=A_9\cup E_7,\\
         &\left(A_5\cup E_7\right)\cap \left(A_9\cup E_7\right)=E_7.
     \end{aligned}\]
     So, $\rho(\sigma_2)$ preserves 
     \[\C^+(\rho(\sigma_5))\cap\C^+(\rho(\sigma_7))\cup E_5\cup E_7=\C\cup \bigcup_{i\in\odd(n)} E_i=\A.\]
     This completes the proof that there is an induced action (by permutations) of $\B_n$ on components of $\A$. 

     Suppose now that $\A\neq C$. Then every element of $\B_n$ which normalizes $\X_n$, permutes elements of the set 
     \[E=\left\{ E_i\right\}_{i\in\odd(n)}.\] Moreover, $\X_n$ is totally symmetric in $\B_n$, hence every permutation of $E$ can be realized via this action. This contradicts the statement (2) of Proposition \ref{prop:ch8:8.1}, which states that the action of $\B_n$ on $\A$ should be cyclic.
\end{proof}
\begin{rem}\label{rem:9.1:expl}
In the second step of the proof of Proposition 9.1 of \cite{CM} the authors use the fact that $\C^+(\rho(\sigma_5))\cup \C^+(\rho(\sigma_7))=\A$ (in their notation this multicurve is called $M$). This is, in general, not true, because the left-hand side contains special curves. However, as we showed above, this inaccuracy can be easily fixed.
\end{rem}
Throughout the rest of the proof we assume that $\C^+(\rho(\X_n))$ has some special components, for otherwise $\rho$ is cyclic by Proposition \ref{all_C_case}.
%%%%%%%%%%%%%%%%%%%%%%%%%%%%%%%%%%%%%%
%%%%%%%%%%%%%%%%%%%%%%%%%%%%%%%%%%%%%%%
\section{Topological configurations of special curves}\label{sec:top_A}
The main goal of this section is to classify the possible configurations of totally symmetric $\X_n$-labelled multicurves in $N=N_{g,b}$, whose every component is special -- see Proposition~\ref{lem:Qclosed}. 
%This proposition is the main reason, why in the 
%assumptions of Theorem \ref{thm:main_pmcg} we have stronger inequality than $g\leq n+1$.

The proof of Proposition \ref{lem:Qclosed} is based on a series of lemmas. In order to simplify the statements of these lemmas, let us introduce some global notation and assumptions for the rest of this section. 
\begin{enumerate}
\item $\A$ is any totally symmetric $\X_n$-labelled multicurve in $N=N_{g,b}$ whose every component is special and $k=|\X_n|=\left\lfloor{n}/{2}\right\rfloor$.
    \item We denote by $A_i$ the set of components of $\A$ with label $\{\sigma_i\}$.
    \item For any component $R$ of $N_\A$, we denote 
    \[l(R)=\{\sigma_i\in \X_n\colon \bndy(R)\cap A_i\neq\emptyset\}.\]
    %\item $Q=N_{g,b}\setminus \A$.
    \item We assume that $k\geq 7$ and
    \begin{equation}\label{eq:special:g:k}
    g\leq 2k+1.
    \end{equation}
    So, in particular,
    \begin{equation}\label{eq:special:g:chi}
    \chi(N)=2-g-b\geq 1-2k-b.
    \end{equation}
\end{enumerate}
%By Corollary \ref{cor:sq_faithful} $\A$ is sq-faithful. Furthermore, the components of $Q$ are in one-to-one correspondence with the components of $N_g\setminus\sq(\A)$. By this observation we can deduce information about $\A$ and $Q$ from $\sq(\A)$ and $N_g\setminus\sq(\A)$.

\begin{lemma}\label{lem:big_small_comp}%[Lemma 10.1 of CM]
    Let $R$ be any component of $N_\A$. Then either $|l(R)|=1$ or $|l(R)|=|\X_n|$.
    %\begin{enumerate}
    %\item $|l(R)|=1$, or 
    %\item $|l(R)|=|\X_n|$ and [WEG] multicurves $\partial R\cap A_i$, $i\in\odd(n)$ are %of the same topological type.
        %\item $\forall_{i\in\odd(n)}\partial R\cap A_i\neq \emptyset$ and multicurves $\partial R\cap A_i$ are of the same topological type, or
        %\item $\partial R\cap A_i\neq \emptyset$ for only a single $i\in\odd(n)$.
    %\end{enumerate}
\end{lemma}
\begin{proof}
Assume that $|l(R)|>1$.

If $|l(R)|<k-1$, then there exist at least 
    \[\binom{k}{|l(R)|}\geq \binom{k}{2}=\frac{k(k-2)}{2}\]
    different components of $N_\A$ (because $\A$ is totally symmetric with respect to  $\X_n$). By Corollary \ref{cor:sq_faithful}, $\A$ is sq-faithful so $N_g\setminus\sq(\A)$ has the same number of connected components as $N_\A=N_{g,b}\setminus \A$ and each of these components has negative Euler characteristic. Hence, inequality \eqref{eq:special:g:k} gives
\[\frac{k(k-2)}{2}\cdot (-1)\geq \chi(N_g)=2-g \geq 1-2k.\]
    This inequality is false as long as $k\geq 6$, so we proved that $|l(R)|\geq k-1$.

    If $|l(R)|=k-1$, then $R$ has at least $k-1$ boundary components which are not external, and there are at least $\binom{k}{k-1}=k$ different components homeomorphic to $R$. Thus, inequality \eqref{eq:special:g:chi} gives
    \[k(3-k)-b\geq  \chi(N_{g,b})\geq 1-2k-b.\]
    This inequality is false as long as $k\geq 5$, so we proved that $|l(R)|=k$.
    %Let 
    %\[k=k_1+k_2+\ldots+k_m,\]
    %where $m$ is the number of topological types of multicurves of the form $\partial R\cap A_i$, $i\in\odd(n)$ and $k_i$ is the number of such multicurves of one particular topological type. If $m>1$, then there are at least 
    %\[\frac{k!}{k_1!k_2!\ldots k_m!}\geq \frac{k!}{1!\cdot (k-1)!}=k\]
    %different permutations of topological types associated to multicurves $\partial R\cap A_i$.  Since $\A$ is totally symmetric, each of these permutations must be realized by some subsurface of $N_\A$ homeomorphic to $R$. Hence, there are at least $k$ such subsurfaces and
    %\[k(2-k)-b\geq  \chi(N_{g,b})=2-g-b\geq -2k-b.\]
    %The above inequality does not hold as long as $k\geq 5$, so we proved that $m=1$.
\end{proof}
\begin{rem}
We follow the terminology from \cite{CM} and we say that the component $R$ of $N_\A$ is \emph{small} if  $|l(R)|=1$. 
%it satisfies the statement (1) of Lemma~\ref{lem:big_small_comp}.  
If $R$ is not small, we call it a \emph{big} component of $N_\A$. By Lemma \ref{lem:big_small_comp}, $|l(R)|=k$ for big components of $N_\A$.
%Note that if $k>1$, then there must exists at least one big component of $N_\A$. In fact, two small components with different labels can not be glued together.
\end{rem}
\begin{lemma}\label{no:big:components} Let $m$ be the number of big components of $N_\A$. Then $m=1$ or $m=2$. Moreover, if $m=2$, then there are no small components of $N_\A$ and each component of $\A$ is non-separating in $N=N_{g,b}$.
%\begin{enumerate}
%\item There is at least one big component of $N_\A$.
%\item There are at most two big components of $N_\A$.
%\item If there are two big components of $N_\A$, then there are no small components of $N_\A$ %and each component of $\A$ is non-separating in $N$.
%\end{enumerate}
\end{lemma}
\begin{proof}
Our global assumption in this section is that there are only special curves in $\A$, hence two small components with different labels can not share a boundary curve from $\A$. This proves that there must be at least one big component, so $m\geq 1$.

If $N_\A$ has at least 3 big components, then Corollary \ref{cor:sq_faithful} implies that the same is true for components of $N_g\setminus\sq(\A)$. Each of these components has at least $k$ boundary components, so inequality \eqref{eq:special:g:k} gives
\[3(2-k)\geq \chi\left(N_g\setminus\sq(\A)\right)= \chi(N_g)=2-g\geq 1-2k.\]
This inequality is false as long as $k\geq 6$, so we proved that $m\leq 2$.

Assume now that $m=2$ and suppose that there is at least one small component of $N_\A$. Since $\A$ is totally symmetric, there must be at least $k=|\X_n|$ small components of $N_\A$. As before, Corollary \ref{cor:sq_faithful} implies that we have the same configuration of components of $N_g\setminus\sq(\A)$. Each of these components has negative Euler characteristics, hence
\[2(2-k)+k\cdot (-1) \geq \chi\left(N_g\setminus\sq(\A)\right)= \chi(N_g)=2-g\geq 1-2k.\]
This inequality is false as along $k\geq 4$. This proves that small components of $N_\A$ can exist only if $m=1$.

Lemma \ref{lem:big_small_comp} and the above analysis implies that if $m=2$, then 
$N_g\setminus\sq(\A)$ has two connected components, If $\sq(\A)$ contains at least one separating curve, then it contains at least $k$ such curves and the number of components of $N_g\setminus\sq(\A)$ is greater than 2. So, all curves in $\sq(\A)$ are non-separating. This implies the same conclusion for $\A$.
\end{proof}
\begin{lemma}\label{special-curves-separating-conf}
If $i\in\odd(n)$ and $A_i$ contains a separating curve $a_i$, then $a_i$ bounds a subsurface $K_i$ homeomorphic either to $N_{2,1}$ or $S_{1,1}$. Moreover, $A_i$ contains at most one separating curve.
\end{lemma}
\begin{proof}
Without loss of generality assume that $A_1$ contains a separating curve $a_1$. By Lemma~\ref{no:big:components},  there is exactly one big component $R$ of $N_\A$. Let $K_1$ be this connected component of $N\setminus a_1$ which is disjoint from $R$. Hence, all components of $N_\A$ contained in $K_1$ are small and labeled with $a_1$. Since $\A$ is totally symmetric, there are small components $\{K_i\}_{i\in\odd(n)}$ all homemorphic to $K_1$, pairwise disjoint and disjoint from $R$. Moreover, the action of $\pmcg(N_{g,b})$ permutes $\{K_i\}_{i\in\odd(n)}$, so there are no external boundary components in $K_i$ and $K_i$ has exactly one boundary component. The big component $R$ has at least $k$ boundary components, so inequality \eqref{eq:special:g:chi} yields
\[\begin{aligned}
    2-k+k\cdot \chi (K_1)-b&\geq \chi(N_\A)\geq 1-2k-b,\\
    \chi (K_1)&\geq \frac{-k-1}{k}=-1-\frac{1}{k}.
\end{aligned}\]
Since $K_1$ has negative Euler characteristics and $|\bndy(K_1)|=1$, we get that $\chi(K_1)=-1$ and $K_1$ is homeomorphic either to $N_{2,1}$ or $S_{1,1}$.

Finally, if $A_1$ contained more than one separating curve, then $N_{g,b}$ would contain 
at least $2k$ pairwise disjoint subsurfaces homeomorphic to $K_1$. Then
\[2-k+2k\cdot (-1) -b \geq \chi(N_\A)\geq 1-2k-b,\]
which contradicts our assumption that $k\geq 7$.
\end{proof}
\begin{lemma}\label{special-curves-nonseparating-conf}
If $i\in\odd(n)$ and there are no separating curves in $A_i$, then $A_i$ consists of a single non-separating curve and there are no small components of $N_\A$.
\end{lemma}
\begin{proof}
Suppose that there are small components of $N_\A$ and let $Q_i$ be the union of all small components labelled with $\{\sigma_i\}$. We assumed that there are no separating curves in $A_i$, hence $R$ and $Q_i$ must share at least two boundary components. Thus
\[\chi(R)\leq 2-2k-b\]
and
\[\begin{aligned}
    1-2k-b&\leq \chi(N)=\chi(N_\A)=\chi(R)+k\cdot \chi (Q_i),\\
    1-2k-b&\leq  2-2k-b+k\cdot (-1),\\
    k&\leq 1.
\end{aligned}\]
This proves, that $N_\A$ can not contain small components.

If $\{R_i\}_{i=1}^m$ is the set of big components of $N_\A$, then Lemma \ref{no:big:components} implies that $m=1$ or $m=2$. If $e_i=|\bndy(R_i)|$ for $i\leq m$, then
\[\sum_{i=1}^me_i=k\cdot 2|A_i|+b\]
and 
\[\begin{aligned}
1-2k-b&\leq \chi(N)=\chi(N_\A)\leq \sum_{i=1}^m (2-e_i),\\
1-2k-b&\leq 2m-2k|A_i|-b,\\
|A_i|&\leq \frac{2m+2k-1}{2k}=\frac{2m-1}{2k}+1\leq \frac{3}{2k}+1.
\end{aligned}\]
We assumed that $k\geq 7$, hence $|A_i|=1$.
\end{proof}

\begin{prop}\label{lem:Qclosed}
    The configuration of components of $N_\A$ is one of the following:
    \begin{enumerate}
        \item Each $A_i$, $i\in\odd(n)$, contains exactly one separating curve which bounds a small component $K_i$ of $N_\A$ homemorphic either to $N_{2,1}$ or $S_{1,1}$. Moreover, $N_\A\setminus \bigcup_{i\in\odd(n)} K_i$ is a big component homeomorphic to $S_{0,k+b}$ or $N_{1,k+b}$. %Note that $A_i$ can still contain non-separating curves, but they must be contained in $\bigcup_{i\in\odd(n)} K_i$.
        \item Each $A_i$, $i\in\odd(n)$, contains exactly one element -- a non-separating curve. In this case there are no small components of $N_\A$. Moreover, $N_\A$ consists of $m=1$ or $m=2$ big components, and
        \begin{enumerate}
            \item if $m=1$, then $N_\A$ is either $S_{0,2k+b}$ or $N_{1,2k+b}$.
            \item if $m=2$, then $N_\A=R_1\cup R_2$ and $\{R_1,R_2\}$ is one of the sets:
            \[\begin{aligned}
              &\{S_{0,k+b_1},S_{0,k+b_2},\},\ \{S_{0,k+b_1},N_{1,k+b_2},\},\ \{S_{0,k+b_1},S_{1,k+b_2},\},\\ 
              &\{S_{0,k+b_1},N_{2,k+b_2}\},\ \{S_{0,k+b_1},N_{3,k+b_2}\},\ \{N_{1,k+b_1},N_{1,k+b_2}\},\\
              &\{N_{1,k+b_1},S_{1,k+b_2}\},\ \{N_{1,k+b_1},N_{2,k+b_2}\},
            \end{aligned}\]
            where $b_1+b_2=b$.
        \end{enumerate}
    \end{enumerate}
\end{prop}
\begin{proof}
It is enough to prove the proposition for $b=0$. The general case follows by replacing $\A$ with $\sq(\A)$.

If $A_i$ contains a separating curve, then by Lemmas \ref{no:big:components} and \ref{special-curves-separating-conf}, $N_\A$ contains exactly one big component $R$ and $A_i$ contains exactly one separating curve which bounds a subsurface $K_i$ homeomorphic to $N_{2,1}$ or $S_{1,1}$. 
Hence,  %Suppose that $\sq(R)$ has $m\cdot k$ boundary components. Then 
\[\begin{aligned}
    1-2k&\leq 2-g=\chi(N_g)=\chi(N_\A)=\chi(R)+k\cdot \chi(K_i),\\
    1&\leq \chi(R)+k.
\end{aligned}\]
Recall that $R$ has  $k$ boundary components, so $R$ is homeomorphic to $S_{0,k}$ or $N_{1,k}$.

Assume now that there are no separating curves in $A_i$. By Lemmas \ref{no:big:components} and \ref{special-curves-nonseparating-conf}, $N_\A$ consists of $m=1$ or $m=2$ big components. Moreover, $|A_i|=1$, so in the first case the big component $R$ has $2k$ boundary curves, and in the second case each of two big components $\{R_1,R_2\}$ has $k$ boundary curves. 

If $m=1$, then 
\[1-2k\leq 2-g=\chi(N_g)=\chi(N_\A)=\chi(R),\]
so $R$ is homeomorphic to $S_{0,2k}$ or $N_{1,2k}$. 

If $m=2$, then 
\[1-2k\leq 2-g=\chi(N_g)=\chi(N_\A)=\chi(R_1)+\chi(R_2).\]
Since $|\bndy(R_1)|=|\bndy(R_2)|=k$, we get the list of possibilities as in the statement.
\end{proof}

%\begin{rem}
%    So it looks that the proof of Lemma 10.1 is the same as in CM. Up to $\chi(N_g)$ of course.
%\end{rem}

%%%%%%%%%%%%%%%%%%%%%%%%%%%%%%%%%%%%%%
%%%%%%%%%%%%%%%%%%%%%%%%%%%%%%%%%%%%%%%
\section{Mutual position of special and normal curves}\label{sec:delta}
%In this section we assume %$n\ge 14$, $g\le 2k+1$, where $k=\lfloor\frac{n}{2}\rfloor$, and 
%that there are some special components of $\C^+(\rho(\X_n))$.
Let's introduce the following elements of $\B_n$:
\[\delta=\sigma_1\sigma_2\cdots\sigma_{n-1},\qquad \sigma_0=\delta\sigma_{n-1}\delta^{-1}.\]
 Now for $0\le i\le n-2$ we have 
 \[\delta\sigma_{i}\delta^{-1}=\sigma_{i+1}.\]
 It follows that $\sigma_0$ satisfies the relations
  \[\sigma_0\sigma_j=\sigma_j\sigma_0\quad\textrm{for\ } j\notin\{1,n-1\},\quad \sigma_i\sigma_0\sigma_i=\sigma_0\sigma_i\sigma_0\quad\textrm{for\ } i\in\{1,n-1\}.\]
  We also introduce the following totally symmetric subsets of $\B_n$:
  \[\Y_n=\delta \X_n\delta^{-1},\quad \Y'_n=\delta^{-1}\X_n\delta_n.\]
  If $n$ is even \[\Y'_n=\Y_n=\{\sigma_0,\sigma_2,\dots,\sigma_{n-2}\},\]
 whereas if $n$ is odd
 \[\Y'_n=\{\sigma_0,\sigma_2,\dots,\sigma_{n-3}\},\quad \Y_n=\{\sigma_2,\sigma_4,\dots,\sigma_{n-1}\}.\]
 Set $\Delta=\rho(\delta)$ and observe that
 \[\C^+(\rho(\Y_n))=\Delta(\C^+(\rho(\X_n))),\quad \left(\textrm{resp.\ }\C^+(\rho(\Y'_n))=\Delta^{-1}(\C^+(\rho(\X_n)))\right)\] is a totally symmetric $\Y_n$-labelled (resp. $\Y'_n$-labelled) multicurve whose every component is  either special or normal. 
\begin{lemma}\label{lem:normal_inv}
The set of normal components of $\C^+(\rho(\X_n))$ is equal to the  set of normal components of $\C^+(\rho(\Y_n))$ or $\C^+(\rho(\Y'_n))$ and is preserved by $\rho(\B_n)$.
\end{lemma}
\begin{proof}
 For $0\le i\le n-1$ let $\C_i=\C^+(\rho(\sigma_i))$.
Let $C_\X$, $C_\Y$ and $C_{\Y'}$ denote the sets of normal components of respectively $\C^+(\rho(\X_n))$, $\C^+(\rho(\Y_n))$ and  $\C^+(\rho(\Y'_n))$. 
We have
\[C_\Y=\C_2\cap\C_4=C_{\Y'}.\]
Consider the element $\tau=\sigma_1\sigma_2\sigma_3\sigma_4$ and note that 
\begin{enumerate}
\item $\tau(\sigma_i)\tau^{-1}=\sigma_{i+1}$ for $i=1,2,3$,
\item $\tau(\sigma_j)\tau^{-1}=\sigma_j$ for $6\le j\le n-1$.
\end{enumerate}
Let $f=\rho(\tau)$. Then 
\[f(C_\X)=f(\C_1\cap\C_3)=\C_2\cap\C_4=C_\Y.\]
On the other hand
\[f(C_\X)=f(\C_7\cap\C_9)=\C_7\cap\C_9=C_\X.\]
Thus $C_\X=C_\Y$.
Since the set $C_\X=C_\Y$ is preserved by $\rho(\sigma_i)$ for $1\le i\le n-1$, it is preserved by $\rho(\B_n)$.
\end{proof}
\begin{rem}\label{def:A:normal}
We denote the union of normal curves by $\A^n$.
\[\A^n=\bigcap_{i=0}^{n-1}\C^+(\rho(\sigma_i)).\]
By Lemma \ref{lem:normal_inv}, $\A^n$ is $\rho(\B_n)$ invariant.
\end{rem}
\begin{rem}\label{def:ai:unique}
Let $A_i$ denote the set of special components of $\C^+(\rho(\sigma_i))$, for $0\le i\le n-1$:
\[A_i=\C^+(\rho(\sigma_i))\setminus\A^n.\]
We have $\Delta(A_i)=A_{i+1}$, where $A_n=A_0$.
By Lemmas \ref{special-curves-separating-conf} and \ref{special-curves-nonseparating-conf},  either $A_i$ contains a unique separating curve $a_i$ for all $0\le i\le n-1$, or $A_i$ consists of a single non-separating curve $a_i$ for all $0\le i\le n-1$. Throughout the rest of this paper $a_i$ denotes either the unique separating curve in $A_i$ (in the former case) or the unique element of $A_i$ (in the latter case).
\end{rem}
\begin{lemma}\label{lem:Delta_orbit}
\begin{enumerate}
    \item If $\sigma_j=\tau\sigma_i\tau^{-1}$ for some $0\le i,j\le n-1$ and $\tau\in\B_n$ then   $\rho(\tau)(a_i)=a_j$. %In particular, $\Delta(A_i)=A_{i+1}$ for $0\le i\le n-1$, where $A_n=A_0$. 
    \item The curves $a_0,a_1,\dots,a_{n-1}$ are pairwise nonisotopic.
\end{enumerate}
\end{lemma}
\begin{proof} (1) We have $\rho(\tau)(\C^+(\rho(\sigma_i)))=\C^+(\rho(\sigma_j))$ and since $\rho(\tau)$ preserves $\A^n$  it maps $A_i$ to $A_j$ and hence $\rho(\tau)(a_i)=a_j$.

(2) The curves $a_0,a_1,\dots,a_{n-1}$ form the orbit $\langle\Delta\rangle.a_1$ 
It has at least $k$ elements, namely $a_1, a_3,\dots,a_{2k-1}$, and its cardinality divides $n$. If follows that $|\langle\Delta\rangle.a_1|=n$ if $n=2k+1$. If $n=2k$ then each of the orbits $\langle\Delta^2\rangle.a_1$ and  $\langle\Delta^2\rangle.a_2$ has $k$ elements and it suffices to show that 
these orbits are disjoint. Suppose that $a_1=a_i$ for some even $i$. There is an element $\tau\in\B_n$ and $j\in\odd(n)$, $j\ne 1$ such that $\tau\sigma_1\tau^{-1}=\sigma_j$ and $\tau\sigma_i\tau^{-1}=\sigma_i$. We have $a_j=\rho(\tau)(a_1)=\rho(\tau)(a_i)=a_i=a_1$, which is a contradiction, because $a_1\ne a_j$.
 \end{proof}

\begin{lemma}\label{lem:component_S}
 There is a component $S$ of $N_{\A^n}$ containing every $a_i$ for $0\le i<n$. Furthermore, $S$ is $\rho(\B_n)$ invariant.  
\end{lemma}
\begin{proof}
For $0\le i\le n-2$ there is an element $\tau_i\in\B_n'$ such that $\tau_i\sigma_i\tau_i^{-1}=\sigma_{i+1}$. We have $\rho(\tau_i)(a_i)=a_{i+1}$ which means that $a_i$ for $0\le i\le n-1$ are in one $\rho(\B_n')$-orbit. By Proposition~\ref{prop:ch8:8.1}  $\rho(\B_n')$ preserves every component of $N_{\A^n}$, which implies that the curves $a_i$ must all be in the same component. This component is preserved by $\rho(\sigma_i)$ for $1\le i\le n-1$, hence it is $\rho(\B_n)$ invariant. 
\end{proof}

By composing $\rho$ with $\cut{\A^n}$ we obtain new representations $\rho_1,\rho_2$.
\[\cut{\A^n}\circ\rho=(\rho_1,\rho_2)\colon\B_n\to\mcg(S)\times\mcg(N_{\A^n}\setminus S).\]
If $S=N_{\A^n}$ then $\rho_2$ does not exist.

\begin{lemma}\label{lem:rho2cyclic}
Suppose $S\ne N_{\A^n}$. Then
 $\rho_2\colon\B_n\to\mcg(N_{\A^n}\setminus S)$ is cyclic.  
\end{lemma}
\begin{proof}
By Proposition \ref{prop:ch8:8.1}  $\rho(\B_n')$ preserves every component of $N_{\A^n}$ and  every component of $\A^n$ with orientation. Thus \[\rho_2(\B_n')\subset\ppmcg(N_{\A^n}\setminus S).\] We define a homomorphism \[\psi\colon\B_{n-2}\to \ppmcg(N_{\A^n}\setminus S)\] by
$\psi(\sigma_i)=\rho_2(\sigma_i\sigma_{n-1}^{-1})$ for $1\le i\le n-3$. 
Since $N_{\A^n}\setminus S$ is a subsurface of $N_\A$, where $\A=\bigcup_{i\in\odd(n)}a_i$, each component of $N_{\A^n}\setminus S$ is either an orientable surface of genus at most $1$ or a nonorientable surface of genus at most $3$ by Lemma \ref{lem:Qclosed}. Since $n-2\ge 12$, $\psi$ is cyclic by Lemma \ref{lem:smallg}. 
We have
$\rho_2(\sigma_1\sigma_{n-1}^{-1})=\rho_2(\sigma_2\sigma_{n-1}^{-1})$ and  
 $\rho_2(\sigma_1)=\rho_2(\sigma_2)$,
which implies that $\rho_2$ is also cyclic.
\end{proof}

%%%%%%%%%%%%%%%%%%%%%%%%%%%%%%%%%%%%%%
%%%%%%%%%%%%%%%%%%%%%%%%%%%%%%%%%%%%%%%
\section{Intersection of separating special curves}\label{sec:separating}
In this section we assume %$g\le 2k+1$, where $k=\lfloor\frac{n}{2}\rfloor$ 
that $\C^+(\rho(\sigma_i))$ contains a separating special curve $a_i$ for $0\le i<n$. 
By Lemma \ref{special-curves-nonseparating-conf}, $a_i$ is the unique separating special curve in $\C^+(\rho(\sigma_i))$ and it bounds a subsurface $K_i$ homeomorphic either to $N_{2,1}$ or $S_{1,1}$. By Lemma \ref{lem:Delta_orbit}, the curves $a_i$ for $0\le i<n$ %form an $\langle\Delta\rangle$-orbit of length $n$. In particular, they 
are pairwise nonisotopic. The purpose of this section is to prove the following proposition.

\begin{prop}\label{prop:ssc}
The sequence $(a_1,\dots,a_{n-1})$ is a chain of separating curves. That is
\begin{enumerate}
\item $K_i$ is a one-holed Klein bottle for $1\le i\le n-1$;
\item $I(a_i,a_{i+1})=2$ for $1\le i\le n-2$;
\item $I(a_i,a_j)=0$ for $|i-j|>1$.
\end{enumerate}
\end{prop}
Note that (3) follows from the fact that $\sigma_i$ commutes with $\sigma_j$ for $|i-j|>1$, so it suffices to prove (1) and (2). The proof is similar to that of  \cite[Proposition 8.0.1]{Castel} and is broken into several lemmas. The idea is to show that $K_1\cap K_2$ is a M\"obius band. % which is achieved by Lemma \ref{lem:K1_cap_K2} below.

For the rest of this section we fix representatives of $a_i$ in tight position, which we denote by the same symbols. %Conse quently, we treat the subsurfaces $K_i$ and $S$ as subsets of $N$, rather than isotopy classes.We also know that  and since $a_i$ and $a_j$ are not isotopic, $K_i$ and $K_j$ are disjoint.
We know that $K_i$ and $K_j$ are disjoint for $|i-j|>1$ because $a_i$ and $a_j$ are not isotopic.
Let \[R=N\setminus\bigcup_{i\in\odd(n)}int(K_i).\] Since we assume $g\le 2\lfloor n/2\rfloor+1$, $R$ is either a surface of genus $0$ or a nonorientable surface of genus $1$. % (see Lemma \ref{lem:Qclosed}). %The former case is possible only if $K_i$ are Klein bottles.

\begin{lemma}\label{lem:homeo_F}
There exists a homeomorphism $F$ of $N$, preserving the boundary components of $N$, such that:
\begin{enumerate}
\item $F(a_i)=a_{4-i}$ for $i\in\{1,2,3\}$, and
\item $F(a_i)=a_i$ for $i\ge 5$.
%\item $G(a_i)=G_{3-i}$ for $i\in\{1,2\}$,
%\item $G(a_i)=a_i$ for $i\ge 4$.
\end{enumerate}
\end{lemma}
\begin{proof}
By the same argument as in the proof of Step 2 of \cite[Proposition 8.0.1]{Castel}, we can define $F$ to be a representative of $\rho(\sigma_1\sigma_2\sigma_1\sigma_3\sigma_2\sigma_1)$.
\end{proof}
We also choose a homeomorphism $\Delta$ of $N$ representing $\rho(\delta)$ such that $\Delta(a_i)=a_{i+1}$ for $i\in\{0,1,\dots,n-1\}$.
Note that $\Delta^2$ and $F$ preserve $R$.

\medskip

Suppose that $S$ is a subsurface of $N$ and $d_1,d_2$ are two (not necessarily distinct) connected components of $\partial{S}$. We denote by $\arcs(S,d_1,d_2)$ the set of all arcs
 properly embedded  in $S$ with one end in $d_1$ and the other end in $d_2$.
Suppose that $\alpha\in\arcs(S,d,d)$. Let $d\cap\alpha=\{P,Q\}$ and denote by $\overline{\alpha}$ a simple closed curve in $S$ composed of $\alpha$ and an arc of $d$ from $P$ to $Q$ (a connected component of $d\setminus\{P,Q\}$). Note that this notation is ambiguous as it depends on the choice of an arc of $d$. We say that $\alpha$ is one-sided (respectively two-sided) if $\overline{\alpha}$ is one-sided (respectively two-sided).
If $\A$ is a union of pairwise disjoint arcs in $S$, then we denote by $S_\A$ the surface obtained by cutting $S$ along $\A$. We say that $\alpha$ is separating in $S$ if $S_\A$ is disconnected.

%In this section we are interested in arcs which are the connected components of $a_i\cap K_{i\pm 1}$ or $a_i\cap R$.

\begin{lemma}\label{lem:arcs13}
Every component of $a_2\cap R$ belongs to $\arcs(R,a_1,a_3)$.    
\end{lemma}
\begin{proof}
We have $K_2\subset R\cup K_1\cup K_3$ and $K_2\not\subset R$ so every component of $a_2\cap R$ is an arc with ends on $a_1\cup a_3$.
It suffices to show that no component of $a_2\cap R$ belongs to $\arcs(R,a_1,a_1)$.
Suppose $\alpha\subset a_2$ and $\alpha\in\arcs(R,a_1,a_1)$. 
Let $\alpha'=F(\alpha)$, where $F$ is the homeomorphism from Lemma \ref{lem:homeo_F} and observe that $\alpha'\subset a_2$ and $\alpha'\in\arcs(R,a_3,a_3)$. Since $a_2$ has no self-intersection, $\alpha$ and $\alpha'$ are disjoint. 
Consider a simple closed curve  $\overline{\alpha}$ composed of $\alpha$ and an arc of $a_1$. Observe that $F(\overline{\alpha})\cap\overline{\alpha}=\emptyset$, and since $R$ has genus at most $1$, $\overline{\alpha}$ must be separating. Let $D$ be the connected component of $R_\alpha$ not containing $a_3$.  Then $F(D)$ does not contain $a_1$, and thus $D\cap F(D)=\emptyset$. It follows that $D$ has genus $0$. But then $D$ must contain at least one component $d$ of $\partial{R}$ different from $a_1$ and $a_3$. But $F$ is fixed on such components, hence $d\subset D\cap F(D)$, a contradiction.
\end{proof}

\begin{lemma}\label{lem:3sep}
Let $a_i, a_j$ be distinct components of $\partial R$ and $\alpha_1,\alpha_2,\alpha_3$ pairwise disjoint elements of $\arcs(R,a_i,a_j)$. Then $\alpha_1\cup\alpha_2\cup\alpha_3$ is separating in $R$.
\end{lemma}
\begin{proof}
Let $R_1=R_{\alpha_1}$ and consider $\alpha_2$ and $\alpha_3$ as arcs in $R_1$ with ends on one component of $\partial R_1$. If $\alpha_2$ is separating in $R_1$, then $\alpha_1\cup\alpha_2$ is separating in $R$. If $\alpha_2$ is nonseparating in $R_1$ then $\alpha_2$ must be one-sided and $R_1$ has genus $1$. But then $R_2=R_{\alpha_1\cup\alpha_2}$ has genus $0$ and $\alpha_3$ has both ends on the same component of $\partial R_2$, so $\alpha_3$ is separating in $R_2$.
\end{proof}

We say that $\{\alpha_1,\alpha_2\}\subset\arcs(R,a_i,a_j)$ is a nonseparating pair of arcs if $\alpha_1,\alpha_2$ are disjoint and $\alpha_1\cup\alpha_2$ is nonseparating in $R$. Note that such pairs cannot exist if $R$ has genus $0$ (see the proof of Lemma \ref{lem:3sep}).

\begin{lemma}\label{lem:pairs}
Suppose that $R$ has genus $1$. Let $a_r, a_s, a_t$ be distinct components of $\partial R$ and $\alpha_1,\alpha_2,\beta_1,\beta_2$ pairwise disjoint arcs such that $\{\alpha_1,\alpha_2\}\subset\arcs(R,a_r,a_s)$, $\{\beta_1,\beta_2\}\subset\arcs(R,a_r,a_t)$ are nonseparating pairs of arcs. Then the ends of $\alpha_1\cup\alpha_2$ alternate the ends of  $\beta_1\cup\beta_2$ along $a_r$ (Figure \ref{fig:alternate}).
\end{lemma}
\begin{figure}[h]
\begin{center}
\includegraphics[width=0.43\customwidth]{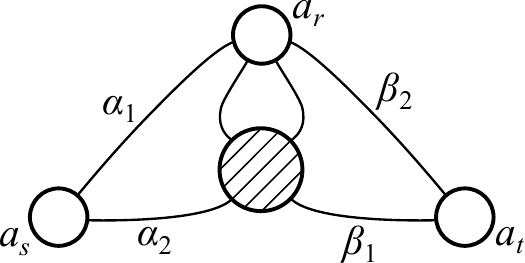}
\caption{The ends of $\alpha_1\cup\alpha_2$ alternate the ends of  $\beta_1\cup\beta_2$ along $a_r$.}\label{fig:alternate} %
\end{center}
\end{figure}
%\begin{figure}
%\includegraphics{Rys/Rys4}
%\caption{\label{fig:alternate}
%\red{[$(d_1,d_2,d_3)\mapsto(a_r,a_s,a_t)$]}\\
%The ends of $\alpha_1\cup\alpha_2$ alternate the ends of  $\beta_1\cup\beta_2$ along $a_r$.}
%\end{figure}
\begin{proof}
Let $A_1,A_2,B_1,B_2$ be the ends in $a_r$ of $\alpha_1,\alpha_2,\beta_1,\beta_2$ respectively. Suppose that $A_1,A_2$ don't alternate $B_1,B_2$ along $a_r$. Then there exist disjoint arcs $\gamma_1,\gamma_2\subset a_r$ such that $A_1,A_2$ are the ends of $\gamma_1$ and $B_1,B_2$ are the ends of $\gamma_2$. Let $\overline{\alpha}$ (resp $\overline{\beta}$) be a simple closed curve in $R$ composed of the arcs $\alpha_1$, $\gamma_1$,  $\alpha_2$ and an arc of $a_s$ from $A'_1$ to $A'_2$ (resp.  $\beta_1$, $\gamma_2$, $\beta_2$ and an arc of $a_t$ from $B'_1$ to $B'_2$). Note that $\overline{\alpha}$ and $\overline{\beta}$ are nonseparating and disjoint. This is a contradiction, because $R$ has genus $1$.
\end{proof}

%\subsection{Proof of Proposition \ref{prop:ssc}.\label{proof:ssc}} 

%\subsection{Determining $T_2\cap S$.}

Following \cite{Castel}, we call every component of  $K_i\cap K_{i+1}$ or $K_i\cap R$ a {\it domain}.
The boundary of a domain $D$ is the union of finitely many arcs of the curves $a_j$, which we call {\it sides} of $D$. A component of $\partial D\cap a_j$ is called $a_j$-side of $D$. 
A domain is called a {\it rectangle} (respectively a {\it hexagon}) if it has four (respectively six) sides and it is homeomorphic to a disc.
%\begin{figure}
%\includegraphics{Rys/Rys5}
%\caption{\label{fig:rect}
%\red{[$(S,R,l)\mapsto(R,D,\gamma)$]}\\
%The case where $\alpha_1\cup\alpha_2$ doesn't separate $R$ (left) %versus the case where $\alpha_1\cup\alpha_2$ separates $R$ (right).}
%\end{figure}
\begin{lemma}\label{lem:rectangles}
Every  component of $K_2\cap R$ is a rectangle.    
\end{lemma}
\begin{proof}
Let $D$ be a component of $K_2\cap R$. By Lemma \ref{lem:arcs13} the $a_2$-sides of $D$ belong to $\arcs(R,a_1,a_3)$. 
Let $\gamma$ be an $a_3$-side of $D$ and $\alpha_1,\alpha_2$ two $a_2$-sides whose ends on $a_3$ are the ends of $\gamma$ (Figure \ref{fig:rect}). 
\begin{figure}[h]
\begin{center}
\includegraphics[width=0.7\customwidth]{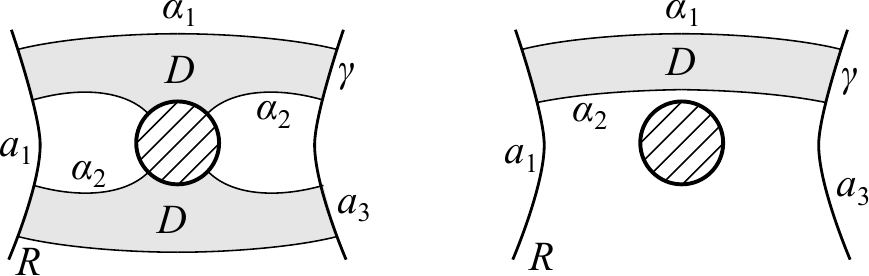}
\caption{The case where $\alpha_1\cup\alpha_2$ doesn't separate $R$ (left) versus the case where $\alpha_1\cup\alpha_2$ separates $R$ (right).}\label{fig:rect} %
\end{center}
\end{figure}
We claim that $\alpha_1\cup\alpha_2$ is separating in $R$.  Suppose that $\{\alpha_1,\alpha_2\}$ is a nonseparating pair of arcs in $R$ (left hand side of Figure \ref{fig:rect}). Then  $\{\Delta^2(\alpha_1),\Delta^2(\alpha_2)\}\subset\arcs(R,a_3,a_5)$ is also a nonseparating pair of arcs in $R$.  By Lemma \ref{lem:pairs}, the ends of $\Delta^2(\alpha_1)\cup\Delta^2(\alpha_2)$ alternate the ends of $\alpha_1\cup\alpha_2$ along $a_3$. It follows that $\gamma$ contains an end of $\Delta^2(\alpha_1)\cup\Delta^2(\alpha_2)$, which is a contradiction because $\Delta^2(\alpha_1)\cup\Delta^2(\alpha_2)\subset K_4$ and $\gamma\subset K_2$. 
Thus $R_{\alpha_1\cup\alpha_2}$ is disconnected and one of its components is $D$ (right hand side of Figure \ref{fig:rect}). We see that $D$ has four sides and it remains to show that it is homeomorphic to a disc.  Indeed, $D$ is not a M\"obius band, because $D\cap\Delta^2(D)=\emptyset$, $D\cup\Delta^2(D)\subset R$ and $R$ cannot contain two disjoint M\"obius bands. \end{proof}

%\subsection{Determining $T_1\cap T_2$.}

%\begin{figure}
%\includegraphics{Rys/Rys1}\\
%\includegraphics{Rys/Rys2}\\
%\includegraphics{Rys/Rys3}
%\caption{\label{fig:regions}Possible components (domains) of $K_1\cap K_2$ as subsets of $K_1$, where $K_1$ is a one-holed Klein bottle.}
%\end{figure}

\begin{lemma}\label{lem:domains}
The components of $K_2\cap K_1$ consist in some rectangles and either
exactly one hexagon, or exactly one M\"obius band whose boundary is the union of an arc of $a_2$ and an arc of $a_1$.
\end{lemma}
 \begin{proof} 
If $K_i$ are one-holed tori, then $K_2\cap K_1$ consist in exactly one hexagon and some rectangles  by the proof of Step 4 of \cite[Proposition 8.0.1]{Castel}, so we consider only the case where $K_i$ are one-holed Klein bottles. 
Then each component $D$ of $K_1\cap K_2$ is one of the shaded domains shown in Figure \ref{fig:regions}.
\begin{figure}[h]
\begin{center}
\includegraphics[width=0.95\customwidth]{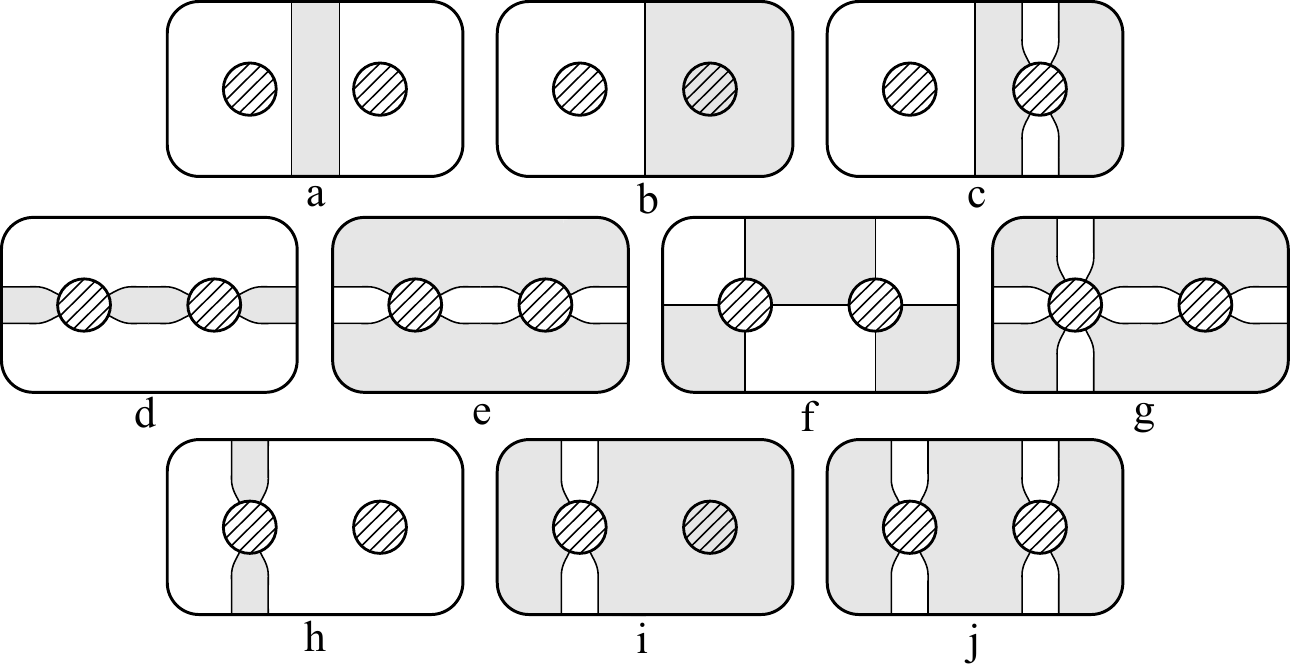}
\caption{Possible components (domains) of $K_1\cap K_2$ as subsets of $K_1$, where $K_1$ is a one-holed Klein bottle.}\label{fig:regions} %
\end{center}
\end{figure}
This can be proved by inspection of three cases: 
\begin{enumerate}
\item  where at least one $a_2$-side of $D$ is separating in $K_1$ (configurations: a, b, c),
\item  where at least one $a_2$-side of $D$ is two-sided and nonseparating in $K_1$ (configurations: d, e, f, g),
\item  where all $a_2$-sides of $D$ are one-sided in $K_1$ (configurations: h, i, j).
\end{enumerate}
Let $\mathcal{D}(K_2)$ denote the set of components of $K_2\cap K_1$, $K_2\cap K_3$ and $K_2\cap R$ and 
notice that $K_2$ is the union of elements of $\mathcal{D}(K_2)$. For $D\in\mathcal{D}(K_2)$ we  determine its contribution $\xi(D)$ to the Euler characteristic of $K_2$ in such a way that
\[\chi(K_2)=\sum_{D\in\mathcal{D}(K_2)}\xi(D).\]
Two distinct domains are either disjoint, or they have a common $a_1$-side, or a common $a_3$-side. Conversely, every ($a_1$ or $a_3$)-side belongs to exactly two domains. It follows that $\xi(D)$ can be defined as
\[\xi(D)=\chi(D)-\frac{1}{2}\chi(D\cap(a_1\cup a_3)),\]
where $D\cap(a_1\cup a_3)$ is the union of ($a_1$ or $a_3$)-sides of $D$. Since the Euler characteristic of a side is $1$, $\chi(D\cap(a_1\cup a_3))$ is simply the number of ($a_1$ or $a_3$)-sides of $D$.
By Lemma \ref{lem:rectangles} the components of $K_2\cap R$ are rectangles with $\xi(D)=0$.
Let us compute the contribution of the possible domains of $K_1\cap K_2$ from Figure \ref{fig:regions}:
\begin{enumerate}
    \item $\xi(D)=0$ if $D$ is a rectangle (configurations: a, d, h),
    \item $\xi(D)=-\frac{1}{2}$ if $D$ is a hexagon (configurations: c, f) or the M\"obius band (configuration~b),
    \item $\xi(D)=-1$ if $D$ is one of the remaining domains (configurations: e, g, i, j).
\end{enumerate}
The homeomorphism $F$ from Lemma \ref{lem:homeo_F} maps the components of $K_1\cap K_2$ on the components 
$K_3\cap K_2$ preserving their contribution: $\xi(F(D))=\xi(D)$. Since $\xi(K_2)=-1$, we see that exactly one component $D$ of $K_1\cap K_2$ is not a rectangle and $\xi(D)=-\frac{1}{2}$, thus $D$ is either a hexagon (configurations: c, f) or the M\"obius band (configuration b). 
\end{proof} 
% \begin{figure}
%\includegraphics{Rys/Rys6}
%\caption{\label{fig:caseA}The hexagons $H\subset K_3\cap K_2$ and  $H'\subset K_3\cap K_4$ and arcs  $\alpha_1,\alpha_2,\alpha_3\subset a_2$, $\beta_1,\beta_2,\beta_3\subset a_4$ from case 1.}
%\end{figure}
\begin{lemma}\label{lem:nohexagons}
No component of $K_2\cap K_1$ is a hexagon.
\end{lemma}
 \begin{proof} 
 Suppose that one component of $K_2\cap K_1$, say $D$ is a hexagon.
First we assume that $K_i$ are one-holed Klein bottles and consider two cases.

\medskip
 \noindent{\bf Case 1: no $a_2$-side of $D$ is separating in $K_1$ (Fig. \ref{fig:regions} f))}. Let
 \[H=F(D)\subset K_2\cap K_3,\quad H'=\Delta^2(D)\subset K_4\cap K_3.\] Observe that $H$ and $H'$ are disjoint hexagons contained in $K_3$, such that the $a_3$-sides of $H$ alternate the $a_3$-sides of $H'$ along $a_3$ (Figure \ref{fig:caseA}). 
 \begin{figure}[h]
\begin{center}
\includegraphics[width=0.35\customwidth]{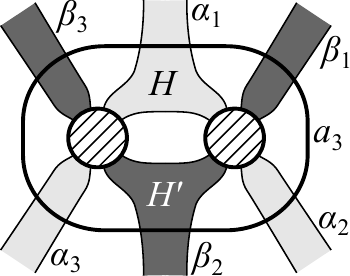}
\caption{The hexagons $H\subset K_3\cap K_2$ and  $H'\subset K_3\cap K_4$ and arcs  $\alpha_1,\alpha_2,\alpha_3\subset a_2$, $\beta_1,\beta_2,\beta_3\subset a_4$ from Case 1.}\label{fig:caseA} %
\end{center}
\end{figure}
 There exist pairwise disjoint arcs $\alpha_1,\alpha_2,\alpha_3\in\arcs(S,a_3,a_1)$ contained in $a_2$, and pairwise disjoint arcs $\beta_1,\beta_2,\beta_3\in\arcs(S,a_3,a_5)$ contained in $a_4$, such that the ends of  $\alpha_1\cup\alpha_2\cup\alpha_3$ alternate the ends of $\beta_1\cup\beta_2\cup\beta_3$ along $a_3$ (Figure \ref{fig:caseA}). By Lemma \ref{lem:3sep} $R'=R\setminus(\alpha_1\cup\alpha_2\cup\alpha_3)$ is disconnected. 
Let $X$ be the component of $R'$ containing $a_5$ and observe that at least one arc $\beta_i$ has one end in $R\setminus R'$ and the other end in $a_5\subset R'$. Such $\beta_i$ has to intersect  $\alpha_1\cup\alpha_2\cup\alpha_3$, which is a contradiction because $\beta_i\subset a_4$ and  $\alpha_1\cup\alpha_2\cup\alpha_3\subset a_2$.

 %\begin{figure}
%\includegraphics{Rys/Rys7}
%\caption{\label{fig:caseB}\red{[$(R.R')\mapsto(E,E')$]}\\
%The hexagons and rectangles from case 2. $H\cup E\subset K_3\cap %K_2$, $H'\cup E'\subset K_3\cap K_4$.}
%\end{figure}

\medskip
 \noindent{\bf Case 2: one $a_2$-side of $D$ is separating in $K_1$ (Fig. \ref{fig:regions} c))}.
  First we show that $K_3\setminus int(K_2)$ is homeomorphic to $K_3\cap K_2$. 
  Since \[K_3\setminus int(K_2)=\Delta(K_2\setminus int(K_1),\] it suffices to show that 
 $K_2\setminus int(K_1)$ is homeomorphic to $K_3\cap K_2$.  We have $K_2\subset K_1\cup K_3\cup R$ and
 \[K_2\setminus int(K_1)=(K_3\cap K_2)\cup(R\cap K_2),\] which is homeomorphic to  $K_3\cap K_2$ since $R\cap K_2$ is a sum of disjoint rectangles.
 
Let $H=F(D)\subset K_2\cap K_3$. Since $K_1\setminus D$ is nonorientable, so is $K_3\setminus H$, and $H\ne K_2\cap K_3$ by the observation above. Thus $K_2\cap K_3$ contains at least one rectangle $E$ with one-sided $a_2$-sides (Fig. \ref{fig:regions} f)). Let $H'=\Delta^2(D)$
and $E'=\Delta^2(F(E))$ and observe that $H'$ and $E'$ are components of $T_3\cap T_4$ and that the $a_3$-sides of $H\cup E$ alternate the $a_3$-sides of $H'\cup E'$ along $a_3$ (Figure \ref{fig:caseB}). 
\begin{figure}[h]
\begin{center}
\includegraphics[width=0.48\customwidth]{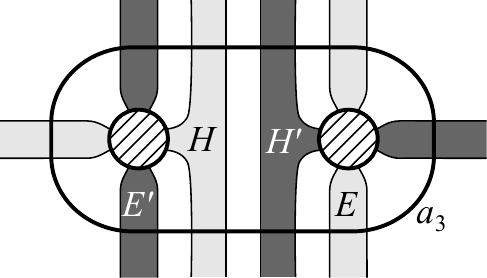}
\caption{The hexagons and rectangles from Case 2. $H\cup E\subset K_3\cap K_2$, $H'\cup E'\subset K_3\cap K_4$.}\label{fig:caseB} %
\end{center}
\end{figure}
This leads to a contradiction as in Case 1.

\medskip
If $K_i$ are one-holed tori, then by the proof of Step 5 of \cite[Proposition 8.0.1]{Castel},  that there are hexagons $H\subset K_2\cap K_3$ and $H'\subset K_4\cap K_3$ such that the sides of $H$ alternate the sides of $H'$ along $a_3$.This leads to a contradiction as in Case 1.
\end{proof}

\begin{lemma}\label{lem:K1_cap_K2}
$K_1\cap K_2$ is a M\"obius band.    
\end{lemma}
\begin{proof} By Lemmas \ref{lem:domains} and \ref{lem:nohexagons} the components of $K_2\cap K_1$ consist in exactly one M\"obius band and some rectangles. In particular, the $K_i$ are one-holed Klein bottles. We have to show that there are no rectangles in  $K_2\cap K_1$.

Let $M_1$ (respectively $M_3$) be the M\"obius band component of $K_1\cap K_2$ (respectively $K_3\cap K_2$) and let 
\[M'_3=\Delta^2(M_1)\subset K_3\cap K_4,\quad  M'_1=\Delta^{-2}(M_3)\subset K_1\cap K_0.\]
Suppose that $K_1\cap K_2$  contains at least one rectangle component $D$. Then $D$ separates $M_1$ from $M'_1$ in $K_1$ and by choosing $D$ appropriately we can assume that all other rectangle components of $K_1\cap K_2$ (if there are any) 
are on the same side of $D$ as $M_1$.
%\begin{figure}
%\includegraphics{Rys/Rys8}
%\caption{\label{fig:step6a}\red{[$(R.R')\mapsto(D,D')$]}\\
%The case where all components of $K_1\cap K_0$ are on the %same side of $D$ versus the case where there are components %of $K_1\cap K_0$ on both sides of $D$.}
%\end{figure}
%
We can also assume that all rectangle components of $K_0\cap K_1$ are on the same side of $D$ as $M_1'$. For otherwise there would be a rectangle $D'\subset K_1\cap K_0$ on the same side of $D$ as $M_1$ and the $a_1$-sides of $M'_1\cup D'$ would alternate the $a_1$-sides of $M_1\cup D$ along $a_1$ (Figure \ref{fig:step6a}). 
\begin{figure}[h]
\begin{center}
\includegraphics[width=0.8\customwidth]{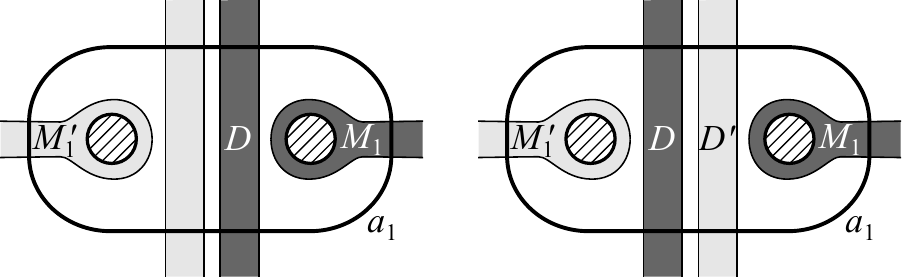}
\caption{The case where all components of $K_1\cap K_0$ are on the same side of $D$ versus the case where there are components of $K_1\cap K_0$ on both sides of $D$.}\label{fig:step6a} %
\end{center}
\end{figure}
This leads to a contradiction, as in the proof of Lemma \ref{lem:nohexagons} (Case 1). 

%\begin{figure}
%\includegraphics{Rys/Rys9}
%\caption{\label{fig:step6b}\red{[$(S,S',R,l)\mapsto(R,R',D,\gamma)$]}\\
% $\alpha_1\cup\alpha_2$ is separating in $R$. $T_2\cap R$ is entirely contained in the component $R'$.}
%\end{figure}
Let $\gamma$ be the $a_2$-side of $D$ separating $K_1\cap K_0$ form $K_1\cap K_2$ and let  $\alpha_1,\alpha_2\in\arcs(R,a_1,a_3)$ be disjoint arcs contained in $a_2$ whose ends on $a_1$ are the same as the ends of $\gamma$ (Figure \ref{fig:step6b}).
\begin{figure}[h]
\begin{center}
\includegraphics[width=0.8\customwidth]{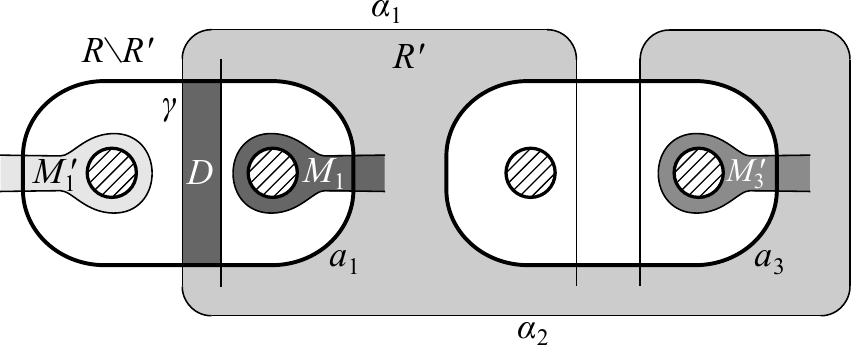}
\caption{$\alpha_1\cup\alpha_2$ is separating in $R$. $T_2\cap R$ is entirely contained in the component $R'$.}\label{fig:step6b} %
\end{center}
\end{figure}
We claim that $R\setminus(\alpha_1\cup\alpha_2)$ is disconnected. For suppose that $\{\alpha_1,\alpha_2\}$ is a nonseparating pair of arcs in $R$ and consider the pair  of arcs $\{\beta_1,\beta_2\}\subset\arcs(R,a_1,a_{n-1}\}$ defined as
$\beta_i=\Delta^{-2}(\alpha_i)$ for $i=1,2$. Since $\beta_1\cup\beta_2$ is contained in $a_0$, it is disjoint form $\alpha_1\cup\alpha_2$. By Lemma \ref{lem:pairs}  the ends of $\alpha_1\cup\alpha_2$ alternate the ends of  $\beta_1\cup\beta_2$ along $a_1$, which contradicts the assumption that $\gamma$ separates  $K_0\cap K_1$ from  $K_2\cap K_1$.
Let $R'$ be the component of $R\setminus(\alpha_1\cup\alpha_2)$ containing $T_2\cap R$. Observe that $K_4\cap R'\ne\emptyset$ and $K_0\cap (R\setminus R')\ne\emptyset$ (Figure \ref{fig:step6b}). The set \[X=R\cap\bigcup_{i=4}^n a_i,\] where $a_n=a_0$ is connected and has nonempty intersection with $R'$ and $R\setminus R'$. It follows that $X$ intersects $\alpha_1\cup\alpha_2\subset a_2$, which is a contradiction. 
We have proved that $K_1\cap K_2$ contains no rectangles, i.e. $K_1\cap K_2=M_1$, which completes the proof. 
\end{proof}

\begin{proof}[Proof of Proposition \ref{prop:ssc}]
By Lemma \ref{lem:K1_cap_K2}, $K_1\cap K_2$ is a M\"obius band whose boundary is the union of an arc of $a_2$ and an arc of $a_1$, hence $I(a_1,a_2)=2$, and by symmetry $I(a_i,a_{i+1})=2$ for all $1\le i\le n-1$. Since $K_1$ contains a M\"obius band, it cannot be a torus, so it is a Klein bottle.
\end{proof}
%%%%%%%%%%%%%%%%%%%%%%%%%%%%%%%%%%%%%%
%%%%%%%%%%%%%%%%%%%%%%%%%%%%%%%%%%%%%%%
\section{Constructing a transvection of $\rho$}\label{sec:transvection}
The assumptions for this section are the same as in Section \ref{sec:delta}, so assume that $k\geq 7$ and $g\leq 2k+1$, where $k=\nhalf$. %  $n\geq 14$, $g\leq n+1$.
%, $k=\lfloor\frac{n}{2}\rfloor$ 
We also assume that there are special components in $\C^+(\X_n)$.

%We denote the surface $N_{g,b}$ by $N$. 
As in Remark \ref{def:A:normal}, we denote by $\A^n$ the union of normal components of $\C^+(\X_n)$. By Lemma \ref{lem:normal_inv}, $\A^n$ is the union of normal components of $\C^+(\Y_n)$ as well.

As in Remark \ref{def:ai:unique}, for any $1\leq i< n$, we denote by $A_i$ the set of special components of $\C^+(\rho(\sigma_i))$ and $a_i$ denotes either the unique separating curve in $\C^+(\rho(\sigma_i))$ or the unique element of $\C^+(\rho(\sigma_i))$ if there are no separating curves in $\C^+(\rho(\sigma_i))$. Moreover, by Proposition~\ref{prop:ssc}, if $a_i$ is separating, then it bounds a Klein bottle $K_i$ with one boundary component $a_i$.

Recall, that by Lemma~\ref{lem:component_S}, there exists a component $S$ of $N_{\A^n}$ containing all $a_i$. If $\A^n$ is empty, then we define $S=N$.
\begin{lemma}\label{lem:R:separating:normal} Suppose that $A_i$ contains a separating curve $a_i$.
\begin{enumerate}
    \item If $A_i$ contains a curve $c\ne a_i$ then $c$ is the unique two-sided and non-separating curve in $K_i$.
    \item $K_i\cap\A^n=\emptyset$. In other words $K_i\subset S$. 
\end{enumerate}
\end{lemma}
\begin{proof}
(1) Suppose that $A_1$ contains a curve $c$, $c\ne a_1$. Then $c$ is two-sided and  Lemma \ref{special-curves-separating-conf} implies that $c$ is non-separating. Since $c$ is disjoint from
$a_i$ for all $i\in\odd(n)$, either $c$ is in $K_j$ for some $j\in\odd(n)$ or $c$ is in $N\setminus\bigcup_{i\in\odd(n)} K_i$. The latter is not possible, as $N\setminus\bigcup_{i\in\odd(n)} K_i$ is either a holed sphere or a holed projective plane (see~Proposition~\ref{lem:Qclosed}), and so it contains no two-sided non-separating curve.
So $c$ is the unique two-sided and non-separating curve in $K_j$ for some $j\in\odd(n)$ and it has non-zero intersection with $a_{j-1}$ and $a_{j+1}$ by Proposition \ref{prop:ssc}. It follows that $j=1$ because  $\sigma_1$ commutes with at least one of $\sigma_{j-1}$ or $\sigma_{j+1}$ for $j\ne 1$.

(2) Let $\tau\in\B_n'$ be such that $\tau\sigma_1\tau^{-1}=\sigma_{3}$. Then $\rho(\tau)(a_1)=a_{3}$ and $\rho(\tau)(K_1)=K_{3}$.
Because $K_1\cap K_3=\emptyset$ and $\rho(\B_n')$ fixes every component of $\A^n$ by Proposition \ref{prop:ch8:8.1}, we must have $K_1\cap\A^n=\emptyset$.
\end{proof}

\begin{lemma}\label{lem:componentR}
Let $\A=\bigcup_{i\in\odd(n)}A_i$. 
    There exists a component $R$ of $S_\A=N_{\A\cup\A^n}$ such that:
    \begin{enumerate}
        \item $a_i\in\bndy(R)$ for all $i\in\odd(n)$,
        \item $R$ is either a holed sphere or a holed projective plane.
    \end{enumerate}
\end{lemma}
\begin{proof}
    If $a_1$ is separating, then let
    \[R=S\setminus\bigcup_{i\in\odd(n)}K_i.\]
    By Lemma \ref{lem:R:separating:normal}, $R$ is a component of $S_\A$ and by Proposition \ref{lem:Qclosed}, $R$ is either a holed sphere or a holed projective plane. This completes the proof in this case, so for the rest of the proof assume that $a_1$ is non-separating.

    By Proposition \ref{lem:Qclosed} $N_\A$ contains at least one big component $Q$ that is either a holed sphere or a holed projective plane. If $Q\cap\A^n=\emptyset$ then $Q$ is also a component of $S_\A$ and we take
    $R=Q$.
    Suppose $Q\cap\A^n\ne\emptyset$ and let $\A'=Q\cap\A^n$. Now we take $R$ to be any component of $Q_{\A'}$ such that $a_1\in\bndy(R)$. Note that $R$ is a component of $S_\A$ and it is either a holed sphere or a holed projective plane. It remains to show that $\bndy(R)$ contains all components of $\A$.
    Let \[l(R)=\{\sigma_i\in\X_n\colon a_i\in\bndy(R)\}\]
    and choose $c\in\bndy(R)\cap\A^n$. Suppose that $\tau\in\B_n'$ is such that $\tau\X_n\tau^{-1}=\X_n$.
    Then $R'=\rho(\tau)(R)$ is a component of $S_\A$ such that
    \[l(R')=\tau l(R)\tau^{-1}\]
    and $c\in\bndy(R')$ because $\rho(\B_n')$ fixes every component of $\A^n$ by Proposition \ref{prop:ch8:8.1}. Suppose $l(R)\ne\X_n$.
    By Lemma \ref{lem:X_n_ts}, every permutation of $\X_n$ can be obtained by conjugation by an element of $\B_n'$, so we can obtain at least $k\ge 7$ different components $R'$ by changing $\tau$. This is a contradiction, because $c$ belongs to at most two different components of $S_\A$.
    \end{proof}

\begin{lemma}\label{lem:13:centralize}
If $i\ge 5$, then $\rho(\sigma_1\sigma_3^{-1})$ fixes $a_i$ and preserves its orientation and sides.
\end{lemma}
\begin{proof}
Let $G$ denote the stabiliser of $a_i$ in $\pmcg(N)$ and $H$ the subgroup of $G$ consisting of elements preserving the orientation and sides of $a_i$. Then $H$ is a normal subgroup of $G$ of index at most $4$. 
Since $\rho(\sigma_1)$, $\rho(\sigma_2)$ and $\rho(\sigma_3)$ commute with $\rho(\sigma_i)$ for $i\ge 5$, they fix $a_i$. It follows that 
    $\rho(\sigma_1)$ and $\rho(\sigma_3)$ are conjugate in $G$ (by an element $\rho(\sigma_2\sigma_1\sigma_3\sigma_2)$). Moreover, $G/H$ is abelian, so
    $\rho(\sigma_1\sigma_3^{-1})\in[G,G]\subseteq H$.
    \end{proof}
\begin{lemma}\label{lem:id_on_subs}
Suppose that $F$ is a subsurface of a surface $M$ and both $F$ and $M$ have negative Euler characteristic. If $f\in\mcg(M)$ is  a periodic mapping class such that $f(F)=F$ and $f|_F=id_F$. Then $f=id$.    
\end{lemma}
\begin{proof}
The proof of \cite[Claim 10.6]{CM} works also for a nonorientable surface. 
%    Alternatively, we could prove it using the orientation double cover.
\end{proof}
\begin{lemma}\label{lem:equal_on_subs}
Suppose that $F$ is a subsurface of a surface $M$ and both $F$ and $M$ have negative Euler characteristic. Let $f\in\mcg(M)$ be such mapping class that:
\begin{enumerate}
    \item $\C^+(f)=\emptyset$;
    \item $f$ fixes the isotopy class of some nontrivial two-sided simple closed curve in $M$;
    %\item $f_i(F)=F$, for $i=1,2$;
    \item $f(F)=F$ and $f|_{F}=id_F$. 
\end{enumerate}
Then $f=id_M$.
\end{lemma}
\begin{proof}
By (1) we know that $f$ is either periodic or pseudo-Anosov or almost pseudo-Anosov on $M$. By (2) and Lemma \ref{apA:nofix}, $f$ is neither pseudo-Anosov, nor almost pseudo-Anosov on $M$, so $f$ is periodic on $M$. Then, (3) and Lemma \ref{lem:id_on_subs} imply that $f=id_M$.
\end{proof}
If $a_1$ is separating, then we define $N(a_i)$ to be the interior of $K_i$ for $0\leq i<n$. Otherwise, we define $N(a_i)$ to be the interior of a regular neighbourhood of $a_i$ in $N$, $0\leq i<n$. In order to unify the arguments regardless of the separability of $a_1$, we introduce the following notation:
\[\begin{aligned}
 N_i^c&=N\setminus N(a_i), &&\text{for $i\in\{1,\ldots,n-1\}$,}\\
 %S_i^c&=S\setminus N(a_i), &&\text{for $i\in\{1,\ldots,n-1\}$}\\
 N_{i,j}^c&=N\setminus \left(N(a_i)\cup N(a_j)\right), &&\text{for $i,j\in\{1,\ldots,n-1\}$,}\\
 %S_{i,j}^c&=S\setminus \left(N(a_i)\cup N(a_j)\right), &&\text{for $i,j\in\{1,\ldots,n-1\}$}\\
 N_{i,j,k}^c&=N\setminus \left(N(a_i)\cup N(a_j)\cup N(a_k)\right), &&\text{for $i,j,k\in\{1,\ldots,n-1\}$}.
\end{aligned}\]
In an analogous way, we define $S_i^c=N_i^c\cap S$ and $S_{i,j}^c=N_{i,j}^c\cap S$. 
%Let $\mathring{K_i}$ denote the interior of $K_i$ and $K_i^c=N\setminus\mathring{K_i}$.   
    \begin{lemma}\label{lem:13:equal:S}
        $\rho(\sigma_1)$ and $\rho(\sigma_3)$ are equal on $S_{1,3}^c$. % S\setminus(T_1\cup T_3)$
    \end{lemma}
    \begin{proof}
Let $R$ be the component of $S_\A=N_{\A\cup\A^n}$ as in the statement of Lemma \ref{lem:componentR}. By Lemma \ref{lem:13:centralize}, $\rho(\sigma_1\sigma_3^{-1})$ fixes $R$ and preserves orientation of the boundary components $a_i$ for $i\ge 5$. 
By Lemma \ref{lem:izoS0N1}, $\rho(\sigma_1\sigma_3^{-1})$ is the identity on $R$. Therefore, $f=\rho(\sigma_1\sigma_3^{-1})$,  $M=S_{1,3}^c$ and $F=R$ satisfy the assumptions of Lemma \ref{lem:equal_on_subs}. Hence, $\rho(\sigma_1)=\rho(\sigma_3)$ on $M=S_{1,3}^c$.
    %By the definition, $\C^+(\rho(\sigma_i))\cap S_i^c=\emptyset$, so we know that $\rho(\sigma_i)$ is either periodic or pseudo-Anosov or almost pseudo-Anosov on $S_i^c$. Since $\rho(\sigma_i)$ commutes with some $\rho(\sigma_j)$, $j\neq i$, it preserves $a_j$, which implies that $\rho(\sigma_i)$ is neither pseudo-Anosov, nor almost pseudo-Anosov on $S_{i}^c$. Therefore $\rho(\sigma_i)$ is periodic on $S_{i}^c$.
%
    %We know that $\rho(\sigma_1)$ and $\rho(\sigma_3)$ commute and each of these elements is periodic on $S_{1,3}^c$. Thus $\rho(\sigma_1\sigma_3^{-1})$ is of finite order on $S_{1,3}^c$. Let $R$ be the component of $S_\A=N_{\A\cup\A^n}$ as in the statement of Lemma \ref{lem:componentR}. By Lemma \ref{lem:13:centralize}, $\rho(\sigma_1\sigma_3^{-1})$ fixes $R$ and preserves orientation of the boundary components $a_i$ for $i\ge 5$. By Lemma \ref{lem:izoS0N1}, $\rho(\sigma_1\sigma_3^{-1})$ is the identity on $R$.  So $\rho(\sigma_1\sigma_3^{-1})$ is a periodic mapping class on $S_{1,3}^c$, which is the identity on $R$. By Lemma \ref{lem:id_on_subs}, it follows that $\rho(\sigma_1\sigma_3^{-1})$ is the identity on $S_{1,3}^c$.
    \end{proof}
    \begin{lemma}\label{lem:13:equal:N}
    $\rho(\sigma_1)$ and $\rho(\sigma_3)$ are equal on $N_{1,3}^c$ %$N\setminus(T_1\cup T_3)$.
    \end{lemma}
    \begin{proof}
        By Lemmas \ref{lem:13:equal:S} and \ref{lem:rho2cyclic}, we know that $\rho(\sigma_1\sigma_3^{-1})$ is the identity on $S_{1,3}^c$ and on $N_{\A^n}\setminus S$. It follows that $\rho(\sigma_1\sigma_3^{-1})$ is a multitwist about normal curves on $N_{1,3}^c$. Similarly, $\rho(\sigma_1\sigma_5^{-1})$ is a multitwist about normal curves on $N_{1,5}^c$. We know that $\sigma_1\sigma_3^{-1}$ and $\sigma_1\sigma_5^{-1}$ are conjugate in $\B_n'$ and every element of $\rho(\B_n')$ fixes every normal curve and preserves the orientation of its regular neighbourhood, by  Proposition \ref{prop:ch8:8.1} and Remark \ref{rem:twist_comm}. It follows that on $N_{1,3,5}^c$ we have $\rho(\sigma_1\sigma_3^{-1})=\rho(\sigma_1\sigma_5^{-1})$ and 
    $\rho(\sigma_3\sigma_5^{-1})$ is the identity on $N_{1,3,5}^c$. We also know that $\rho(\sigma_3\sigma_5^{-1})$ is a multitwist about normal curves on $N_{3,5}^c$, which implies that in fact $\rho(\sigma_3\sigma_5^{-1})$ is the identity on $N_{3,5}^c$. By symmetry, the same conclusion is true for $\rho(\sigma_1\sigma_3^{-1})$ on $N_{1,3}^c$.
    \end{proof}
    %\begin{lemma}[WEG]\label{lem:eq_on_subs}
%Let $\{i,j,k\}\subset\{1,\dots,n-1\}$ be such that $i+1<j<k-1$ and let $F=S_{j,k}^c$.
%If $h\in\mcg(N)$ satisfies $h|_{N_i^c}=id$  and $h|_{F}=id$, then $h=id$.
%\end{lemma}
%\begin{proof}
%    The same proof as in [CM, Claim 10.8]. From the assumption we know that $\C^+(h)$ is empty, but $h$ is not pseudo Anosov. Thus $h$ has finite order and by Lemma \ref{lem:id_on_subs} we have $h=id$.
%\end{proof}
\begin{lemma}\label{lem:propagate:tau}
Let $p,q\in\{1,\ldots,n-1\}$ be such that $|p-q|>1$, and let $\tau\in\mcg(N)$ be such mapping class that:
    \begin{enumerate}
        %\item $\tau(S)=S$ and $\tau$ is periodic on $S$;
        \item[(i)] $\tau$ commutes with $\rho(\sigma_q)$, and 
        \item[(ii)] $\tau|_{N_p^c}=\rho(\sigma_p)|_{N_p^c}$ and $\tau|_{N_q^c}=\rho(\sigma_q)|_{N_q^c}$.
    \end{enumerate}
    Then for any $r\in\{1,\ldots, n-1\}$ such that $|r-p|>1$ and $|r-q|>1$ we have
    \begin{enumerate}
    \item $\tau$ commutes with $\rho(\sigma_r)$, and
    \item $\tau|_{N_r^c}=\rho(\sigma_r)|_{N_r^c}$.
    \end{enumerate}
\end{lemma}
\begin{proof}
From (ii) we know that 
\[\C^+(\tau|_S)\subset \C^+\left(\rho(\sigma_p)|_{S_p^c}\right) \cup \C^+\left(\rho(\sigma_q)|_{S_q^c}\right)=\emptyset \]
and $\tau(a_r)=a_r$, so $\tau$ is periodic on $S$ by Lemma \ref{apA:nofix}. 

    Let $\eta\in\B_n$ be such that
\[\eta\sigma_p\eta^{-1}=\sigma_p,\quad \eta\sigma_q\eta^{-1}=\sigma_r,\quad \eta\sigma_r\eta^{-1}=\sigma_q.\]
Then $h=\rho(\eta)$ satisfies
\[h(a_p)=a_p,\quad h(a_q)=a_r,\quad h(a_r)=a_q,\]
and $\tau'=h\tau h^{-1}$ commutes with $\rho(\sigma_r)$. Moreover, 
\[\tau'|_{N_p^c}=\rho(\sigma_p)|_{N_p^c},\quad \tau'|_{N_r^c}=\rho(\sigma_r)|_{N_r^c},\]
so it suffices to show that $\tau'=\tau$. 

First we prove that $\tau'$ commutes with $\tau$. Let $f=[\tau,\tau']$ and $F=S_{q,r}^c$. We have $\tau'|_F=\rho(\sigma_r)|_F$ and $\tau|_F=\rho(\sigma_q)|_F$, so $f|_{F}=id_F$. We also have $\tau'|_{N_p^c}=\tau|_{N_p^c}=\rho(\sigma_p)|_{N_p^c}$, so $f|_{N_p^c}=id_{N_p^c}$. This implies that $\C^+(f)=\emptyset$, so $f$, $F$ and $M=N$ satisfy the assumptions of Lemma \ref{lem:equal_on_subs} and $f=id$ on $N$.

It follows that $\tau'\tau^{-1}|_{S}$ is a finite order element, which is the identity on $S_p^c$, so by Lemma \ref{lem:id_on_subs} we have
$\tau'\tau^{-1}|_S=id_S$. We also have $\tau'\tau^{-1}|_{N_p^c}=id_{N_p^c}$, so $f=\tau'\tau^{-1}$, $F=S$ and $M=N$ satisfy the assumptions of Lemma \ref{lem:equal_on_subs}. Hence, $\tau=\tau'$.
\end{proof}
    \begin{prop}\label{prop:existence:tau}
        There exists $\tau\in\pmcg(N)$ such that
\begin{enumerate}
    \item $\tau$ commutes with $\rho(\sigma_i)$ for $1\le i<n$, and
    \item $\tau^{-1}\rho(\sigma_i)$ is the identity on $N_i^c$ for $1\le i<n$.
\end{enumerate}
    \end{prop}
    \begin{proof}
Let $h_1$ be a homeomorphism of $N_1^c$ representing $\rho(\sigma_1)|_{N_1^c}$. By Lemma \ref{lem:13:equal:N}, $\rho(\sigma_1)=\rho(\sigma_3)$ on $N_{1,3}^c$, so by \cite[Claim 10.7]{CM} there exists a homeomorphism $h_3$ of $N_3^c$ representing  $\rho(\sigma_3)|_{N_3^c}$ and equal to $h_1$ on $N_{1,3}^c$.
Let $h$ be the homeomorphism of $N$ defined by $h|_{N_1^c}=h_1$ and $h|_{N_3^c}=h_3$.
We define $\tau$ to be the isotopy class of $h$. 

Let us prove, that $\tau$ commutes with $\rho(\sigma_3)$. If $f=[\tau,\rho(\sigma_3)]$, then %\tau\rho(\sigma_3)\tau^{-1}\rho(\sigma_3)^{-1}
\[f|_{N_1^c}=id_{N_1^c},\quad f|_{N_3^c}=id_{N_3^c},\]
so $\C^+(f)=\emptyset$. Hence, $f$, $F=N_1^c$ and $M=N$ satisfy the assumptions of Lemma \ref{lem:equal_on_subs} and $f=id_N$.

We proved that (1) and (2) hold for $i=3$ and that $\tau$ satisfies the assumptions of Lemma~\ref{lem:propagate:tau}. If we apply this lemma with $p=1$, $q=3$ and any $r\in\{5,\dots,n-1\}$, we obtain (1) and (2) for any $i\in\{5,\dots,n-1\}$. We then apply the same lemma with $p=1$, $q=6$ and $r=4$ to get the statement for $i=4$. Finally, we apply the lemma with $p=4$, $q=6$ and $r\in\{1,2\}$ to complete the proof.

  \end{proof}

%%%%%%%%%%%%%%%%%%%%%%%%%%%%%%%%%%%%%%
%%%%%%%%%%%%%%%%%%%%%%%%%%%%%%%%%%%%%%%
\section{Finishing the proof of Theorem \ref{thm:main_pmcg}}\label{sec:finish}
A transvection $\rho'=\rho^{\tau^{-1}}$ of the original representation $\rho$ by $\tau^{-1}$ constructed in Proposition~\ref{prop:existence:tau} yields a new homomorphism such that 
\[\rho'(\sigma_i)|_{N_i^c}=id,\quad \text{for $i=1,\ldots,n-1$}.\]
In order to finish the proof of Theorem \ref{thm:main_pmcg} it is enough to show that $\rho'$ is either a standard twist representation or a \ctr. We do this in Propositions~\ref{prop:braid_on_N31_nonsep} and \ref{prop:braid_on_N31} below.

%\subsection{The case where $a_i$ are separating}
%In this subsection we assume that $a_i$ bounds a one-holed Klein bottle $K_i$ for $0\le i\le n-1$. 
\begin{prop}\label{prop:braid_on_N31_nonsep}
Suppose that $A_i$ consists of a single non-separating curve and $\rho(\sigma_i)$ is the identity on $N_i^c$ for $i=1,\dots,n-1$. Then $\rho$ is a standard twist representation.
\end{prop}
\begin{proof}
Since $\rho(\sigma_i)$ is the identity on $N_i^c$, we know that 
\[\rho(\sigma_i)=t_{a_i}^{k_i}\]
is a power of a Dehn twist along $a_i$ for $1\le i<n$.  We have  \[t_{a_i}^{k_i}t_{a_{i+1}}^{k_{i+1}}t_{a_i}^{k_i}=t_{a_{i+1}}^{k_{i+1}}t_{a_i}^{k_i}t_{a_{i+1}}^{k_{i+1}},\]
which by \cite[Proposition 4.8]{StukowFM} implies $I(a_i,a_{i+1})=1$. Moreover, if local orientations used to define the twist agree, then  $k_i=k_{i+1}\in\{-1,1\}$. For $|i-j|>1$ we have $t_{a_i}^{k_i}t_{a_j}^{k_j}=t_{a_i}^{k_j}t_{a_i}^{k_i}$ which implies $I(a_i,a_j)=0$ by \cite[Proposition 4.7]{StukowFM}. It follows that $a_i$ for $1\le i\le n-1$ form a chain of non-separting curves and $\rho'$ is a standard twist representation . 
\end{proof}
\begin{prop}\label{prop:braid_on_N31}
Suppose that $A_i$ contains a separating curve $a_i$ bounding a Klein bottle with a hole $K_i$ and  $\rho(\sigma_i)$ is the identity on $N_i^c$ for $i=1,\dots,n-1$. Then $\rho$ is a \ctr.
\end{prop}
\begin{proof}
By Proposition \ref{prop:ssc}, $a_i$ for $1\le i\le n-1$ form a chain of separating curves.
For $i=1,\dots,n-2$ let $\mu_{i+1}$ be the core of  the M\"obius band $K_{i}\cap K_{i+1}$, $\mu_1$ the core of $K_{1}\setminus K_{2}$ and $\mu_n$ the core of $K_{n-1}\setminus K_{n-2}$.
Choose an orientation of a regular neighbourhood of the union of $a_i$ for $i=1,\dots,n-1$ and let $t_{a_i}$ be the right Dehn twist about $a_i$ with respect to the chosen orientation. Consider the following elements of $\mcg(K_i,\partial K_i)$:
\begin{itemize}
\item  the crosscap transposition  $u_i$  swapping $\mu_i$ and $\mu_{i+1}$, and such that $u_i^2=t_{a_i}$,
\item  the Dehn twist $t_{b_i}$ about the unique nonseparating two-sided curve $b_i$ on $K_i$, 
\end{itemize}
where we assume that $t_{b_i}$ are right handed with respect to some orientation of a regular neighbourhood of the union of $b_i$ for $i=1,\dots,n-1$. The group $\mcg(K_i,\partial K_i)$ admits the following presentation (see \cite{StukowFM}):
\[\langle u_i, t_{b_i}\,\mid\, t_{b_i}u_it_{b_i}=u_i\rangle.\]
Since  $\rho(\sigma_i)$ is the identity on the complement of $K_i$, it belongs to $\mcg(K_i,\partial K_i)$ treated as a subgroup of $\ppmcg(N_{g,p})$. Therefore \[\rho(\sigma_i)=t_{b_i}^{m_i}u_i^{n_i}\] for some integers $m_i, n_i$ and $1\le i\le n-1$. It suffices to show $m_1=m_2=0$ and $n_1=n_2\in\{-1,1\}$. 
From the braid relation $\sigma_2\sigma_1\sigma_2^{-1}=\sigma_1^{-1}\sigma_2\sigma_1$ we have \[\rho(\sigma_2)(a_1)=\rho(\sigma_1)^{-1}(a_2)\] and since $b_i$ is the unique nonseparating two-sided curve in $K_i$ also
\begin{equation}\label{eq:sigma_b}
    \rho(\sigma_2)(b_1)=\rho(\sigma_1)^{-1}(b_2).
\end{equation}
Observe that $\rho(\sigma_i)$ is in the twist subgroup if and only if $n_i$ is even. Since $\sigma_1$ and $\sigma_2$ are conjugate, and the twist subgroup is normal, $n_1\equiv n_2\pmod{2}$.
We consider two cases.

\medskip
\noindent{\bf Case 1:} $n_1$ and $n_2$ are even, $n_i=2k_i$ for $i=1,2$. Then $\rho(\sigma_i)=t_{b_i}^{m_i}t_{a_i}^{k_i}$ for $i=1,2$.
Let us compare the homology classes $[\rho(\sigma_2)(b_1)]$ and $[\rho(\sigma_1)^{-1}(b_2)]$ in $H_1(N,\Z)$.
By abuse of notation we use the same symbol for a mapping class and for the induced automorphism of  $H_1(N,\Z)$. Note that $t_{a_i}$ induces the identity on $H_1(N,\Z)$. Let $x_i=[\mu_i]$ for $i=1,2,3$. For an appropriate choice of orientations we have $[b_1]=x_1+x_2$, $[b_2]=x_2+x_3$ and 
\begin{align*}
[\rho(\sigma_2)(b_1)]&=t_{b_2}^{m_2}[b_1]=[b_1]-m_2[b_2]=x_1+x_2-m_2(x_2+x_3),\\    
[\rho(\sigma_1)^{-1}(b_2)]&=t_{b_1}^{-m_1}[b_2]=[b_2]-m_1[b_1]=x_2+x_3-m_1(x_1+x_2).    
\end{align*}
From \eqref{eq:sigma_b} we have \[[\rho(\sigma_2)(b_1)]=\pm[\rho(\sigma_1)^{-1}(b_2)]\] which implies 
 $m_1=m_2\in\{-1,1\}$.
We can assume $m_1=m_2=1$ (otherwise we replace $\rho(\sigma_i)$ by  $\rho(\sigma_i)^{-1}$). 
Now \eqref{eq:sigma_b} is 
\[t_{a_2}^{k_2}t_{b_2}(b_1)=t_{a_1}^{-k_1}t_{b_1}^{-1}(b_2).\]
Note that $t_{b_2}(b_1)=t_{b_1}^{-1}(b_2)$ and denote this curve by $c$ (Figure \ref{fig:K12:c}). 
\begin{figure}[h]
\begin{center}
\includegraphics[width=0.55\customwidth]{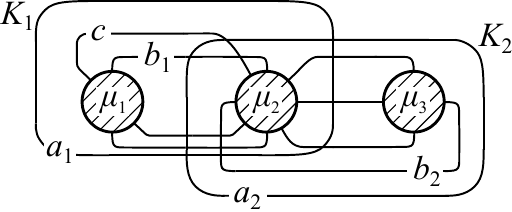}
\caption{Curve $c=t_{b_2}(b_1)=t_{b_1}^{-1}(b_2)$ in $K_1\cup K_2$.}\label{fig:K12:c} %
\end{center}
\end{figure}
If $k_1\ne 0$ then 
\[I(t_{a_1}^{-k_1}(c),b_2)=I(c,t_{a_1}^{k_1}(b_2))=4|k_1|-1\ge 3\]
(Figure \ref{fig:K12:2}), whereas
\[I(t_{a_2}^{k_2}(c),b_2)=I(c,t_{a_2}^{-k_2}(b_2))=I(c,b_2)=1.\]
\begin{figure}[h]
\begin{center}
\includegraphics[width=0.45\customwidth]{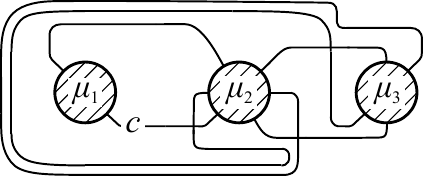}
\caption{Intersection of $c$ and $t_{a_1}(b_2)$.}\label{fig:K12:2} %
\end{center}
\end{figure}
We see that $t_{a_2}^{k_2}(c)=t_{a_1}^{-k_1}(c)$ implies $k_1=0$, and by symmetry also $k_2=0$.
We have obtained $\rho_1(\sigma_1)=t_{b_1}^{\pm 1}$ which contradicts the assumption $a_1\in\C^+(\rho(\sigma_1))$.

\medskip

\noindent{\bf Case 2:} $n_1$ and $n_2$ are odd, $n_i=2k_i+1$ for $i=1,2$. Now \[\rho(\sigma_i)=t_{b_i}^{m_i}t_{a_i}^{k_i}u_i=u_it_{b_i}^{-m_i}t_{a_i}^{k_i},\] for $i=1,2$ and
\begin{align*}
[\rho(\sigma_2)(b_1)]&=u_2t_{b_2}^{-m_2}[b_1]=x_1+x_3+m_2(x_2+x_3),\\    
[\rho(\sigma_1)^{-1}(b_2)]&=u_1^{-1}t_{b_1}^{-m_1}[b_2]=x_1+x_3-m_1(x_1+x_2),   
\end{align*}
which implies $m_1=m_2=0$. 
We have $\rho(\sigma_i)=u_i^{n_i}$ for $i=1,2$ and now it is not difficult to see that \eqref{eq:sigma_b} holds only for $n_1=n_2\in\{-1,1\}$. Indeed, we have
\[
I(u_2^{n_2}(b_1),b_2)=1,\quad
 I(u_1^{-n_1}(b_2),b_2)=2|n_1|-1,
 \]
which implies $|n_1|=1$, and by symmetry also $|n_2|=1$. It must be $n_1=n_2$ because
$u_1(b_2)\ne u_2(b_1)$. 
\end{proof}
  
\begin{rem}\label{rem:conjugate}
Comparing the proofs of Propositions \ref{prop:braid_on_N31_nonsep} and \ref{prop:braid_on_N31}, one might expect that, as is the case of twists, if two elements of the form 
\(t_{b_i}^{m_i}u_i^{n_i},\ t_{b_j}^{m_j}u_j^{n_j}\)
are conjugate, then $m_i=m_j$ and $n_i=n_j$. This is not true in general -- if $b_i$ and $u_i$ are as in the above proof and $k\in\mathbb{Z}$, then 
\[t_{b_i}^{2k}u_i=t_{b_i}^{k}t_{b_i}^{k}u_i=t_{b_i}^ku_it_{b_i}^{-k},\] which means that 
   $t_{b_i}^{2k}u_i$ is a crosscap transposition conjugate to $u_i$.    
   On the other hand, the proof of Proposition \ref{prop:braid_on_N31} implies that $u_i$ is the only element of $\mcg(K_i,\partial K_i)$ which satisfies the braid relation with $u_{i+1}$.

\end{rem}

%%%%%%%%%%%%%%%%%%%%%%%%%%%%%%%%%%%%%%
%%%%%%%%%%%%%%%%%%%%%%%%%%%%%%%%%%%%%%%
%%%%%%%%%%%%%%%%%%%%%%%%%%%%%%%%%%%%%%%
\section{Proofs of Theorems \ref{thm:main_relmcg} and \ref{thm:app}}\label{sec:app}
We deduce Theorem \ref{thm:main_relmcg} from Theorem \ref{thm:main_pmcg} using the argument from \cite{Castel} and the central exact sequence
\[1\to\langle t_d\mid d\in\bndy(N)\rangle\to\mcg(N,\partial N)\stackrel{\forget{\partial N}}{\longrightarrow}\ppmcg(N)\to 1.\]
The following lemma is proved in \cite{Castel}.
\begin{lemma}[Lemma 3.1.4 of \cite{Castel}]\label{lem:central_ext}
Let $1\to N\to G\stackrel{\psi}{\to}\widetilde{G}\to 1$ be a central exact sequence of groups. Let $n\ge 3$ and let $\rho$ and $\rho'$ be two homomorphisms from $\B_n$ to $G$ such that $\psi\circ\rho=\psi\circ\rho'$. Then
\begin{enumerate}
    \item $\rho'$ is a transvection of $\rho$ by a central element,
    \item $\rho$ is cyclic if and only if $\psi\circ\rho$ is cyclic.
\end{enumerate}
\end{lemma}
Theorem \ref{thm:main_relmcg} follows from Theorem \ref{thm:main_pmcg} and the next proposition.
For brevity, we use the abbreviations \emph{s.t.r} and \emph{c.t.r} for standard twist representation and {\ctr} respectively.
\begin{prop}
Let $n\ge 3$, $N=N_{g,b}$ and $\rho\colon\B_n\to\mcg(N,\partial N)$ be a homomorphism. Then
\begin{enumerate}
    \item  $\rho$ is cyclic if and only if $\forget{\partial N}\circ\rho$ is cyclic;
    \item  $\rho$ is a transvection of an s.t.r  if and only if $\forget{\partial N}\circ\rho$ is transvection of an s.t.r;
    \item  $\rho$ is a transvection of a c.t.r if and only if $\forget{\partial N}\circ\rho$ is a transvection of a c.t.r.
\end{enumerate}
\end{prop}
\begin{proof}
    Part (1) follows from Lemma \ref{lem:central_ext}. The proof of (2) is the same as the proof of Proposition 3.2.2 (iv) in \cite{Castel}, so we only prove (3).

    If $\rho$ is a transvection of a c.t.r, then clearly $\forget{\partial N}\circ\rho$ is also a transvection of a c.t.r.
    Conversely, suppose that $\forget{\partial N}\circ\rho$ is a transvection by some element $\tau$ of a c.t.r determined by a chain of separating curves $C$. Then $\tau$ fixes each curve of $C$ and  preserves orientation of its regular neighbourhood. Let $\tau'$ be any element of $\mcg(N,\partial N)$ such that $\forget{\partial N}(\tau')=\tau$. Then $\tau'$ also fixes each curve of $C$ and  preserves orientation of its regular neighbourhood.
    Let $f'$ be the trasvection of $\theta_C\colon\B_n\to\mcg(N,\partial N)$ by ${\tau'}$. We have 
    $\forget{\partial N}\circ\rho'=\forget{\partial N}\circ\rho$ and by Lemma \ref{lem:central_ext} we know that $\rho$ is a transvection of $\rho'$ by a central element. It follows that $\rho$ is a transvection of a c.t.r.  
\end{proof}

To prove Theorem \ref{thm:app} we need generators of the twist subgroup.

\begin{lemma}\label{lem:twist_gens}
    Let $N=N_{g,b}$ for $g\ge 4$ and $b\ge 0$, and let 
    $C=(a_1,\dots,a_{g-1})$ be the chain of nonseparating curves in $N$ from Figure \ref{r01}. There exists a finite subset 
    $X\subset\curv^+(N)$ with the following properties.
    \begin{enumerate}
        \item For every $c\in X$ there exist $a_i,a_j\in C$ such that $I(c,a_i)=0$ and $I(c,a_j)=1$.
        \item $\T(N,\partial T)$ is generated by 
        \[\{t_{a_i}\mid a_i\in C\}\cup\{t_c\mid c\in X\}\cup\{t_d\mid d\in\bndy(N)\}.\]
    \end{enumerate}
\end{lemma}
\begin{proof}
    For $b\le 1$ the lemma follows from the main result of \cite{Omori}. In this case $X$ consists of two curves.

    Let $b\ge 2$. It is proved in \cite{StukowOJM} that  $\T(N,\partial{N})$ is generated by \[\{t_{a_i}\mid a_i\in C\}\cup\{t_c\mid c\in X'\}\cup\{t_d\mid d\in\bndy(N)\},\] where 
    $X'\subset\curv^+(N)$ is  finite, but  $X'$ contains curves which don't satisfy (1). However, all such curves lie on a subsurface homeomorphic to $N_{g,1}$, which also contains $C$, so they can be replaced by the two curves from \cite{Omori} to obtain a set $X$ satisfying (1) and (2). 
\end{proof}

\begin{proof}[Proof of Theorem \ref{thm:app}]
Let $N=N_{g,b}$, $g\ge 13$, $b\ge 0$. We have \[\T(N)=\forget{\partial N}(\T(N,\partial{N})),\]
so it suffices to prove the theorem for a homomorphism
    \[\varphi\colon\T(N,\partial N)\to\pmcg(N_{g',b'}),\] where $g'\le2\left\lfloor\frac{g-1}{2}\right\rfloor$. Let $C$ and $X$ be as in Lemma \ref{lem:twist_gens} and let
    \[\rho_C\colon\B_g\to \T(N,\partial N)\] be the homomorphism defined by $\rho_C(\sigma_i)=t_{a_i}$ for $a_i\in C$. Consider \[\varphi\circ\rho_C\colon\B_g\to \pmcg(N_{g',b'}).\] 
    Since $g'\le2\left\lfloor\frac{g-1}{2}\right\rfloor$, $\varphi\circ\rho$ is neither a transvection of an s.t.r nor a transvection of a c.t.r by Remark \ref{rem:existence}. 
    
    Suppose $g\ge 14$. Then $\varphi\circ\rho_C$ is cyclic  by Theorem \ref{thm:main_relmcg}. Let $A=\varphi(t_{a_i})$ for all $a_i\in C$. Let $c\in X$ be arbitrary and let $a_i,a_j$ be as in (1) of Lemma \ref{lem:twist_gens}. We have $t_ct_{a_i}=t_{a_i}t_c$ and $t_ct_{a_j}t_c=t_{a_j}t_ct_{a_j}$ for appropriate choice of $t_c$. It follows that $\varphi(t_c)$ commutes with $A$ and also satisfies the braid relation with $A$, so $\varphi(t_c)=A$. It follows that $\varphi(\T(N,\partial{N}))$ is generated by $A$ and $\varphi(t_d)$ for $d\in\bndy(N)$ and hence it is abelian.  But  $\T(N,\partial{N})$ is perfect for $g\ge 7$ \cite{StukowOJM}, so $\varphi$ is trivial. 

    If $g=13$ then we add one curve to $C$ to obtain a chain $C'$ of length $13$ (see Remark \ref{rem:g_chain}) 
    and we apply Theorem \ref{thm:main_relmcg} to \[\varphi\circ\rho_{C'}\colon\B_{14}\to \pmcg(N_{g',b'}).\]
    We have $g'\le 12$, so $\varphi\circ\rho_{C'}$ is neither a transvection of an s.t.r nor a transvection of a c.t.r by Remark \ref{rem:existence}, so it must be cyclic. It follows that $\varphi$ is trivial by the same argument as above.
\end{proof}

%%%%%%%%%%%%%%%%%%%%%%%%%%%%%%%%%%%%%%%%%%%%%%%%%%%%%%%
%%%  Bibliography  %%%
%%%%%%%%%%%%%%%%%%%%%%%

\end{document}